\RequirePackage{fix-cm}
\documentclass[smallextended,numbook,runningheads]{svjour3}
\smartqed  
\usepackage{graphicx}
\usepackage{amsmath,amssymb}
\usepackage{graphicx,graphics,color,booktabs}
\usepackage[titletoc]{appendix}

\journalname{...}
\date{ \phantom{b} \vspace{45mm}\phantom{e}}
\def\makeheadbox{\hspace{-4mm}Version of 6 October 2020}

\makeatletter
\newcommand{\settitle}{\@maketitle}
\makeatother

\newcommand\bfB{{\mathbf B}}
\newcommand\bfC{{\mathbf C}}
\newcommand\bfD{{\mathbf D}}

\newcommand\bfG{{\mathbf G}}

\newcommand\bfK{{\mathbf K}}

\newcommand\bfM{{\mathbf M}}

\newcommand\bfV{{\mathbf V}}

\newcommand\bfphi{{\boldsymbol \vphi}}
\newcommand\bfpsi{{\boldsymbol \psi}}

\newcommand\bfrho{{\boldsymbol \rho}}
\newcommand\bfsigma{{\boldsymbol \sigma}}

\newcommand\calD{{\mathcal D}}
\newcommand\calE{{\mathcal E}}
\newcommand\calF{{\mathcal F}}
\newcommand\calH{{\mathcal H}}

\newcommand\calS{{\mathcal S}}
\newcommand\calT{{\mathcal T}}

\renewcommand{\d}{\textrm{d}}
\DeclareMathOperator{\diff}{\frac{\textrm{d}}{\textrm{d}t}}
\newcommand{\eps}{\varepsilon}
\newcommand{\Ga}{\Gamma}
\newcommand{\half}{{\textstyle \frac12}}
\newcommand{\Half}{\frac12}
\newcommand{\iu}{\textnormal{i}}
\def \e {\textnormal{e}}
\newcommand{\inv}{^{-1}}

\newcommand{\N}{\mathbb{N}}
\newcommand{\nb}{\nabla}
\newcommand{\Om}{\Omega}
\newcommand{\pa}{\partial}
\newcommand{\R}{\mathbb{R}}
\DeclareMathOperator{\re}{\mathrm{Re}}

\def \to {\rightarrow}
\newcommand{\vphi}{\varphi}

\newcommand\andquad{\qquad \hbox{ and } \qquad }

\newcommand\qin{\qquad \hbox{ in }}

\newcommand\bfE{\boldsymbol{E}}
\newcommand\bfg{\boldsymbol{g}}
\newcommand\bfj{\boldsymbol{j}}
\newcommand\bfJ{\boldsymbol{J}}
\newcommand\bfH{\boldsymbol{H}}
\DeclareMathOperator{\curl}{curl}
\DeclareMathOperator{\diverg}{div}
\renewcommand{\div}{\diverg}
\newcommand{\dt}{{\mathit{\Delta}t}}
\newcommand{\E}{E}
\newcommand{\Fb}{\calF_h^{\textnormal{bnd}}}
\newcommand{\Fi}{\calF_h^{\textnormal{int}}}
\DeclareMathOperator{\ga}{\gamma}
\renewcommand{\H}{H}
\newcommand{\Hdiv}{\mathcal{H}_\Gamma}

\newcommand{\Ht}{H_\times} 
\newcommand{\J}{J}
\newcommand{\la}{[}
\newcommand{\TG}{\Ga} 
\renewcommand{\P}{I_h}
\newcommand{\ra}{]}

\newcommand{\vpp}{\vect{\vphi}{\psi}} 
\newcommand{\av}[1]{\{\!\!\{#1\}\!\!\}}
\newcommand{\jp}[1]{[\![#1]\!]}
\newcommand{\m}[1]{m(#1)}
\newcommand{\vect}[2]{\binom{#1}{#2}}

\newcommand{\C}{\mu\inv}

\newcommand{\QED}{\hfill $\square$}

\begin{document}
	
\title{Stable and convergent fully discrete interior--exterior coupling of Maxwell's equations}

\titlerunning{Interior--exterior coupling for Maxwell's equations}        

\author{Bal\'{a}zs Kov\'{a}cs \and
	Christian Lubich 
}


\institute{B. Kov\'{a}cs \at
	Mathematisches Institut, University of T\"{u}bingen,\\
	Auf der Morgenstelle 10, 72076 T\"{u}bingen, Germany \\
	\email{kovacs@na.uni-tuebingen.de}
	\and
	Ch. Lubich \at
	Mathematisches Institut, University of T\"{u}bingen,\\
	Auf der Morgenstelle 10, 72076 T\"{u}bingen, Germany \\
	\email{lubich@na.uni-tuebingen.de}
}

\date{}

\maketitle

\begin{abstract}
	Maxwell's equations are considered with transparent boundary conditions, for initial conditions and inhomogeneity having support in a bounded, not necessarily convex three-dimensional domain or in a collection of such domains.  The numerical method only involves the interior domain and its boundary. The transparent boundary conditions are imposed via a time-dependent boundary integral operator that is shown to satisfy a coercivity property.  The stability of the numerical method relies on this coercivity and on an anti-symmetric structure of the discretized equations that is inherited from a weak first-order formulation of the continuous equations. The method proposed here uses  a discontinuous Galerkin method and the leapfrog scheme in the interior and is coupled to boundary elements and convolution quadrature on the boundary. The method is explicit in the interior and implicit on the boundary. Stability and convergence of the spatial semidiscretization  are proven, and with a computationally simple stabilization term, this is also shown for the full discretization.
	\keywords{transparent boundary conditions \and Calderon operator \and discontinuous Galerkin \and boundary elements \and leapfrog scheme \and convolution quadrature}
	\subclass{35Q61 \and 65M60 \and 65M38 \and 65M12 \and 65R20}
\end{abstract}

\section{Introduction}

Maxwell's equations on the whole three-dimensional space are considered with initial conditions and inhomogeneity having support in a bounded domain that is not required to be convex  (or in a finite collection of such domains). The study of such problems leads to transparent boundary conditions, which yield the restriction of the solution to the domain. Such boundary conditions are nonlocal in space and time, for both acoustic wave equations and Maxwell's equations. There is a vast literature to tackle this problem in general for wave equations: fast algorithms for exact, nonlocal boundary conditions on a ball \cite{GroteKeller,Hagstrom}, local absorbing boundary conditions \cite{EngquistMajda,HagstromMarOrGivoli},  perfectly matched layers, which were originally considered for electromagnetism in \cite{Berenger}, and numerical coupling with time-dependent boundary integral operators \cite{abboud2011coupling,leapfrog}. All the above approaches, except the last one, are inadequate for non-convex domains. The local methods fail because waves may leave and re-enter a non-convex domain. Inclusion of a non-convex domain in a larger convex domain is computationally undesirable in situations such as a cavity or an antenna-like structure or a far-spread non-connected collection of small domains.  

The main objective of the present work is to transfer the programme of \cite{leapfrog} from acoustic wave equations to Maxwell's equations: to propose and analyze a provably stable and convergent fully discrete numerical method that couples discretizations in the interior and on the boundary, without requiring convexity of the domain. Like Abboud {\it et al.} \cite{abboud2011coupling} (and later also \cite{leapfrog}) for the acoustic wave equation, we start from a symmetrized weak first-order formulation of Maxwell's equations. In the interior this is discretized by a discontinuous Galerkin (dG) method in space \cite{DiPietroErn,HesthavenWarburton,HochbruckSturm} together with the explicit leapfrog scheme in time \cite{HairerLubichWanner}. The boundary integral terms are discretized by standard boundary element methods in space and by convolution quadrature (CQ) in time \cite{LubichCQ,Lubich-multistep}. This yields a coupled method that is explicit in the interior and implicit on the boundary. The choice of a CQ time discretization of the boundary integral operators is essential for our analysis, and to a lesser extent also the choice of the leapfrog scheme in the interior. However, our approach is not specific to the chosen space discretizations which could, in particular, be replaced by conformal edge elements \cite{Hiptmair-acta}.

While the general approach of this paper is clearly based on \cite{leapfrog}, it should be emphasized that the appropriate boundary integral formulation requires a careful study of the time-harmonic Maxwell's equation. This is based on \cite{BuffaCostabelSheen_traces,BuffaHiptmair,BuffaHiptmairvonPetersdorffSchwab,Costabel2004,BallaniBanjaiSauterVeit},
with  special attention to the appropriate trace space on the boundary and to the corresponding duality. Due to the analogue of Green's formula for Maxwell's equations, the duality naturally turns out to be an anti-symmetric pairing. 

The Calderon operator for Maxwell's equation, which arises in the boundary integral equation formulation of the transparent boundary conditions, differs from the acoustic case to a large extent, and therefore the study of its coercivity property is an important and nontrivial point. Similarly to the acoustic case, the continuous-time and discrete-time coercivity is obtained from the Laplace-domain coercivity  using an operator-valued version, given in \cite{leapfrog}, of the classical Herglotz theorem \cite{Herglotz}. Both the second and first order formulation of Maxwell's equations are used.

The spatial semi-discretization of the symmetrized weak first-order formulation of Maxwell's equations has formally the same matrix--vector formulation as for the acoustic wave equation studied in \cite{leapfrog}, with the same coercivity property of the Calderon operator. Because of this structural similarity, the stability results of \cite{leapfrog}, which are shown using the matrix--vector setting, remain valid for the Maxwell case without any modification. On the other hand, their translation to the functional analytic setting differs to a great extent. Therefore further care is required in the consistency analysis.

In Section~\ref{section: recap helmholtz}  we recapitulate the basic theory for Maxwell's equation in the Laplace domain. Based on Buffa and Hiptmair \cite{BuffaHiptmair}, and further on \cite{BuffaHiptmairvonPetersdorffSchwab,BallaniBanjaiSauterVeit}, we describe the right boundary space, which allows for a rigorous boundary integral formulation for Maxwell's equations. Then the boundary integral operators are obtained in a usual way from the single and double layer potentials.

In Section \ref{section: calderon} we prove the crucial technical result of the present work, a coercivity property of the Calderon operator for Maxwell's equation in the Laplace domain. This property translates to the continuous-time Maxwell's equations later, in Section~\ref{subsection: Calderon op for Maxwell's eqn}, via an operator-valued Herglotz theorem.

In Section~\ref{section: boundary int form} we study the interior--exterior coupling of Maxwell's equations, resulting in an interior problem coupled to an equation on the boundary with the Calderon operator. We derive a first order symmetric weak formulation, which is the Maxwell analogue of the formulation of \cite{abboud2011coupling} for the acoustic wave equation. Together with the continuous-time version of the coercivity property of the Calderon operator, this formulation allows us to derive an energy estimate. Later on this analysis is transfered to the semi-discrete and fully discrete settings.

Section~\ref{section: discretization} presents the details of the discretization methods: In space we use discontinuous Galerkin finite elements with centered fluxes in the domain \cite{DiPietroErn,HesthavenWarburton}, coupled to continuous linear boundary elements on the surface.
Time discretization is done by the leapfrog scheme in the interior domain, while on the boundary we use convolution quadrature based on the second-order backward differentiation formula. An extra term stabilizes the coupling, just as for the acoustic wave equation \cite{leapfrog}.
The matrix--vector formulation of the semidiscrete problem has the same anti-symmetric structure and the same coercivity property as for the acoustic wave equation, and therefore the stability results shown in \cite{leapfrog} can be reused here.

In Sections~\ref{section: semidiscrete results} and \ref{section: fully discrete results} we revise the parts of the results and proofs of \cite{leapfrog} where they differ from the acoustic case, which is mainly in the estimate of the consistency error. Finally, we arrive at the convergence error bounds for the semi- and full discretizations.

\smallskip
To our knowledge, the proposed numerical discretizations in this paper are the first provably stable and convergent semi- and full discretizations to interior--exterior coupling of Maxwell's equations. We believe that the presented analysis and the techniques, which we share with \cite{leapfrog}, can be extended further: to other discretization techniques for the domain, such as edge element methods \cite{Hiptmair-acta}, higher order discontinuous Galerkin methods, and different time discretizations in the domain, together with higher order  Runge--Kutta based convolution quadratures on the boundary \cite{BanjaiLubichMelenk}. 

For ease of presentation we consider only constant permeability and permittivity. However, it is only important that the permeability and permittivity are constant in the exterior domain and in a neighbourhood of the boundary. In the interior these coefficients may be space-dependent and discontinuous. In the latter case the equations can be discretized in space with the dG method as described in \cite{HochbruckSturm}.

In this paper we focus on the appropriate boundary integral formulation and on the numerical analysis of the proposed numerical methods. Numerical experiments are intended to be presented in subsequent work.

\smallskip
Concerning notation, we use the convention that vectors in $\R^3$ are denoted by italic letters (such as $u, E, H,\dots$), whereas the corresponding boldface letters are used for finite element nodal vectors in $\R^N$, where $N$ is the (large) number of discretization nodes. Hence, any boldface letters appearing in this paper refer to the matrix--vector formulation of spatially discretized equations. Functions defined in the domain $\Omega$ are denoted by letters from the Roman alphabet,  while functions defined on the boundary $\Gamma$ are denoted by Greek letters.

\section{Recap: the time-harmonic Maxwell's equation and its boundary integral operators}
\label{section: recap helmholtz}
\subsection{Preliminaries and notation}

Let us consider the \emph{time-harmonic Maxwell's equation}, obtained as the Laplace transform of the second order Maxwell's equation (with constant permeability $\mu$ and permittivity $\eps$):
\begin{equation}
\label{eq: time-harmonic Maxwell }
\begin{aligned}
\eps\mu s^2u + \curl  \curl u =&\ 0 \qin \R^3 \setminus \Ga ,
\end{aligned}
\end{equation}
where $\Ga$ is the boundary of a bounded piecewise smooth domain (or a finite collection of such domains) $\Om\subset \R^3$, not necessarily convex, with exterior normal $\nu$.

We shortly recall some useful concepts and formulas regarding the above problem, based on \cite{BuffaHiptmair} and \cite{KirschHettlich}. For the usual trace we will use the notation $\ga$. The \emph{tangential} and \emph{magnetic} traces are defined, respectively, as
\begin{equation*}
\gamma_T v = v|_\Ga \times \nu, \andquad \gamma_N v = (s\inv \curl v)|_\Ga \times \nu .
\end{equation*}
These traces are also often called \emph{Dirichlet trace} and \emph{Neumann trace}, motivated by the analogue of Green's formula for Maxwell's equations (for sufficiently regular functions):
\begin{equation}
\label{eq: Green}
\begin{aligned}
\int_\Om \bigl( w \cdot \curl v - \curl w \cdot v \bigr)\d x =&\ \int_\Ga (\ga w \times \nu) \cdot \ga v \,\d\sigma \\
=&\ \int_\Ga -(\ga w \times \ga v) \cdot \nu\, \d\sigma .
\end{aligned}
\end{equation}
We introduce an  important notation, the
\begin{equation*}
\textnormal{\it anti-symmetric pairing on $L^2(\Ga)$}: \quad\ \la \ga w , \ga v \ra_\Ga = \int_\Ga (\ga w \times \nu) \cdot \ga v \,\d\sigma ,
\end{equation*}
which appears on the right-hand side of \eqref{eq: Green}.
We note that the relation $\la \ga w , \ga v \ra_\Ga = \la \ga_T w , \ga_T v \ra_\Ga$ holds, cf.\ \cite{BuffaHiptmair,KirschHettlich}.

Let us now set $w=s\inv \curl u$, which provides
\begin{equation*}
\frac1s \int_\Om \bigl( \curl u \cdot \curl v - \curl \curl u \cdot v \bigr) \d x=
\la \ga_N u , \ga_T v \ra_\TG .
\end{equation*}
Moreover, if $u$ satisfies \eqref{eq: time-harmonic Maxwell }
and $v=u$, then
\begin{equation}
\label{eq: Green v=u}
\int_\Om \bigl( s\inv |\curl u|^2 + \eps\mu s |u|^2 \bigr) \d x = \la \ga_N u , \ga_T u \ra_\TG .
\end{equation}

\subsection{Function spaces}
We collect some results on function spaces, which will play an important role later on. All of the results in the present subsection can be found in Section~2 of \cite{BuffaHiptmair}.

Let us start by recalling the usual definition of the Sobolev space corresponding to the $\curl$ operator:
\begin{equation*}
H(\curl,\Om) = \big\{ v \in L^2(\Om)^3 \,:\, \curl v \in L^2(\Om)^3 \big\} ,
\end{equation*}
with corresponding norm $\|v\|_{H(\curl,\Om)}^2 = \|v\|_{L^2(\Om)^3}^2 + \|\curl v\|_{L^2(\Om)^3}^2$.

Clearly, the above integral relations hold for functions $v,w \in H(\curl,\Om)$. 

Now we are turning to trace spaces. However, even though $\ga_T: H(\curl,\Om) \to H^{-1/2}(\Ga)$ is a continuous mapping, $H^{-1/2}(\Ga)$ is not the right choice for boundary integral operators. As it was emphasized by Buffa and Hiptmair \cite{BuffaHiptmair}: the study of the continuous mapping $\ga_T : H(\curl,\Om) \to H^{-1/2}(\Ga)$ is \emph{``actually sufficient for the understanding of homogeneous boundary conditions for fields in the Hilbert space context. However, to impose meaningful non-homogeneous boundary conditions or, even more important, to lay the foundations for boundary integral equations we need to identify a proper trace space''}\footnote{Quoted from Buffa and Hiptmair, \cite{BuffaHiptmair}, Section~2.2.}. In the following, we briefly summarize the definition of such a trace space, together with some related results.

The Hilbert space $\Ht^p(\Ga)$ collects the $\ga_T$ traces of $H^{p+1/2}(\Om)$ functions, for $p \in (0,1)$, i.e., $\Ht^p(\Ga)=\ga_T(H^{p+1/2}(\Om))$. This space is equipped with an inner product  such that $\ga_T : H^{p+1/2}(\Om) \to \Ht^p(\Ga)$ is continuous. In particular, the space $\Ht^{1/2}(\Ga)=\ga_T(H^1(\Om))$ has the dual space $\Ht^{-1/2}(\Ga)$, defined  with respect to the (extended) duality $\la\cdot,\cdot\ra_\TG$.

Then, the above mentioned \emph{proper trace space} is given as:
\begin{equation*}
\Hdiv = H_{\times}^{-1/2}(\div_\Ga,\Ga) = \big\{ w \in \Ht^{-1/2}(\Ga) \,:\, \div_\Ga w \in H^{-1/2}(\Ga) \big\} ,
\end{equation*}
with norm
\begin{equation}\label{Hdiv-norm}
\|w\|_{\Hdiv}^2 = \|w\|_{\Ht^{-1/2}(\Ga)}^2 + \|\div_\Ga w\|_{H^{-1/2}(\Ga)}^2 .
\end{equation}

The tangential trace satisfies the following analogue of the trace theorem.
\begin{lemma}[\cite{BuffaCostabelSheen_traces}, Theorem 4.1]
	\label{lemma: trace ineq}
	The trace operator $\ga_T : H(\curl,\Om) \to \Hdiv$ is continuous and surjective.
\end{lemma}

The following lemma clarifies the role of the anti-symmetric pairing $\la \cdot,\cdot \ra_\TG$.\begin{lemma}[\cite{BuffaCostabelSheen_traces}, Lemma 5.6 and \cite{BuffaHiptmair}, Theorem 2]
	\label{lemma: Hdiv-dual}
	The pairing $\la \cdot,\cdot \ra_\TG$ can be extended to a continuous bilinear form on $\Hdiv$. With this pairing the space $\Hdiv$ becomes its own dual.
\end{lemma}

The above results clearly point out that a natural choice of trace space is $\big(\Hdiv,\la \cdot,\cdot \ra_\TG\big)$, which fits perfectly to the analogue of Green's formula \eqref{eq: Green} and to the boundary integral formulation of Maxwell's equations. This trace space is appropriate for the analysis of boundary integral operators.

\subsection{Boundary integral operators}
\label{subsection: boundary integral op}

On potentials and boundary integral operators we follow Buffa and Hiptmair \cite{BuffaHiptmair}, and we also refer to \cite{BuffaHiptmairvonPetersdorffSchwab,Costabel2004}.

The usual boundary integral potentials for the time-harmonic Maxwell's equation
\begin{equation*}
\eps\mu s^2u + \curl \curl u = 0 \qin \R^3 \setminus \Ga
\end{equation*}
are obtained, based on \cite{BuffaHiptmair} and \cite{BallaniBanjaiSauterVeit}: the (electric) \emph{single layer potential} is given, for $x \in \R^3\setminus \Ga$, as
\begin{equation*}
\calS(s)\vphi(x) = -s \int_\Ga G(s,x-y) \vphi(y) \d y + s\inv \frac{1}{\eps\mu} \nb \int_\Ga G(s,x-y) \div_\Ga \vphi(y) \d y,
\end{equation*}
while the (electric) \emph{double layer potential} is given, for $x \in \R^3\setminus \Ga$, as
\begin{equation*}
\calD(s)\vphi(x) = \curl \int_\Ga G(s,x-y) \vphi(y) \d y,
\end{equation*}
where the fundamental solution is given, for $z\in\R^3$, as
\begin{equation*}
G(s,z) = \frac{e^{-s\sqrt{\eps\mu}|z|}}{4\pi |z|}.
\end{equation*}

The solution then has the representation
\begin{equation}
\label{eq: potential representation of solution}
u = \calS(s)\vphi + \calD(s)\psi,  \qquad x \in \R^3\setminus \Ga ,
\end{equation}
where
\begin{equation}
\label{eq: Laplace domain boundary term equalities}
\vphi = \jp{\gamma_N u} = \jp{\ga_T (s\inv \curl u)} \andquad \psi = \jp{\gamma_T u}.
\end{equation}
Here $\jp{\gamma v} = \gamma^- v - \gamma^+ v$ denotes the jumps in the boundary traces. A further notation is the average of the inner and outer traces on the boundary: $\av{\gamma v} = \half (\gamma^- v + \gamma^+ v)$. On vectors both operations are acting componentwise.

For every $\vphi\in\Hdiv$ and $\psi\in\Hdiv$, formula \eqref{eq: potential representation of solution} defines
$u\in H_{\mathrm{loc}}(\curl,\R^3\setminus\Gamma)$. Because of the jump relations
\begin{align*}
\jp{\gamma_N \circ \calS(s)} =&\ \hbox{Id}, &  \jp{\gamma_N \circ \calD(s)} =&\ 0,
\\
\jp{\gamma_T \circ \calS(s)} =&\ 0, & \jp{\gamma_N \circ \calD(s)} =&\ \hbox{Id},
\end{align*}
$\vphi$ and $\psi$ are reconstructed from $u$ by \eqref{eq: Laplace domain boundary term equalities}.

Let us now define the boundary integral operators. As opposed to the general second order elliptic case, due to additional symmetries of the problem, they reduce to two operators $V$ and $K$, see  \cite[Section~5]{BuffaHiptmair}. They satisfy
\begin{align*}
V(s) =&\ \av{\gamma_T \circ \calS(s)} = \av{\gamma_N \circ \calD(s)} , \\
K(s) =&\ \av{\gamma_T \circ \calD(s)} = \av{\gamma_N \circ \calS(s)} . 
\end{align*}
In \cite[Section~5]{BuffaHiptmair} the continuity of these operators was proven, without giving an explicit dependence on $s$. Such bounds are crucial in the analysis later, therefore we now show
$s$-explicit estimates for the boundary integral operators. Our result is based on \cite{BallaniBanjaiSauterVeit}.
\begin{lemma}
	\label{lemma: int operator s bounds}
	For $\re s \geq \sigma>0$ the boundary integral operators $V(s), K(s): \Hdiv \to \Hdiv$ are bounded as
	\begin{align*}
	\|V(s)\| \leq  C(\sigma) |s|^2  \andquad
	\|K(s)\| \leq  C(\sigma) |s|^2 .
	\end{align*}
\end{lemma}
\begin{proof}
	These estimates can be shown by adapting the arguments of \cite[Section~4.2]{BallaniBanjaiSauterVeit}. In particular, by using the anti-symmetric pairing $\la\cdot,\cdot\ra_\TG$ instead of the usual $L^2(\Ga)$ inner product, the results of \cite[Theorem~4.4]{BallaniBanjaiSauterVeit} transfer from $H^{-1/2}(\div_\Ga,\Ga) \to H^{-1/2}(\curl_\Ga,\Ga)$ to the estimates stated here.
	\QED\end{proof}

Furthermore, using the potential representation of the solution \eqref{eq: potential representation of solution}, the averages of the traces can be expressed using the operators $V$ and $K$ in the following way:
\begin{equation}
\label{eq: boundary int ops}
\begin{aligned}
\av{\gamma_T u}
=&\ \av{ \gamma_T \calS(s)\vphi } + \av{ \gamma_T \calD(s)\psi } \\
=&\ V(s)\vphi + K(s)\psi , \qquad \andquad \\[2mm]
\av{\gamma_N u}
=&\ \av{ \gamma_N \calS(s)\vphi } + \av{ \gamma_N \calD(s)\psi } \\
=&\ K(s)\vphi + V(s)\psi .
\end{aligned}
\end{equation}

\section{Coercivity of a Calderon operator for the time-harmonic Maxwell's equation}
\label{section: calderon}
An important role will be played by the following operator on $\Hdiv\times\Hdiv$, to which we refer as a \emph{Calderon operator}:
\begin{equation}
\label{def-B}
B(s)= \C
\left(
\begin{array}{cc}
V(s)  &  K(s) \\
-K(s) & -V(s) \\
\end{array}
\right) .
\end{equation}
The extra factor $\mu^{-1}$ appears unmotivated here, but will turn out to be convenient later.
This operator satisfies the following coercivity result, which is the key lemma of this paper.

\begin{lemma}
	\label{lemma: coercivity}
	There exists  $\beta>0$ such that the Calderon operator \eqref{def-B} satisfies
	\begin{align*}
	\re \biggl\la \vect{\vphi}{\psi}, B(s) \vect{\vphi}{\psi} \biggr\ra_\TG
	\geq \beta \,\m{s}  \Big( (\eps\mu)\inv \|s\inv\vphi\|_{\Hdiv}^2 + \|s\inv\psi\|_{\Hdiv}^2 \Big)
	\end{align*}
	for $\re s>0$ and for all $\vphi, \psi \in \Hdiv$, with
	$
	\m{s} = \min\{ 1, |s|^{2}\eps\mu \} \re s .
	$
\end{lemma}
\begin{proof}
	The proof has a structure similar to the proof of the corresponding result for the acoustic Helmholtz equation \cite[Lemma~3.1]{leapfrog}, although it now uses a different functional-analytic setting.  The structural similarity becomes possible thanks to the anti-symmetric duality pairing that replaces the symmetric duality pairing of the acoustic case.
	
	For given $\vphi, \psi \in \Hdiv$, we define $u\in H(\curl,\R^3\setminus\Gamma)$ by the representation formula \eqref{eq: potential representation of solution}. We can then express $\vphi$ and $\psi$ in terms of $u$ by \eqref{eq: Laplace domain boundary term equalities}. We note that  \eqref{eq: boundary int ops} yields
	\begin{equation*}
	B(s) \vect{\vphi}{\psi} = \C \vect{\av{\gamma_T u}}{- \av{\gamma_N u}}.
	\end{equation*}
	Now, using the properties of the  anti-symmetric pairing $\la\cdot,\cdot\ra_\TG$ (acting componentwise on $\Hdiv \times \Hdiv$), the analogue of Green's formula \eqref{eq: Green v=u} and using the definition of the traces, we obtain
	\begin{align*}
	\mu  \ \biggl\la \vect{\vphi}{\psi}, B(s) \vect{\vphi}{\psi} \biggr\ra_\TG
	=&\ \ \la \jp{\ga_N u} , \av{\ga_T u} \ra_\TG + \la \jp{\ga_T u} , -\av{\ga_N u} \ra_\TG \\
	=&\ \la \ga_N^- u , \ga_T^- u \ra_\TG - \la \ga_N^+ u , \ga_T^+ u \ra_\TG\\
	=&\ s \ \Big(\|s\inv \curl u\|_{L^2(\R^3\setminus\Ga)}^2 + \eps\mu \|u\|_{L^2(\R^3\setminus\Ga)}^2 \Big).
	\end{align*}
	We further obtain
	\begin{align*}
	& \|\psi\|_{\Hdiv}^2 = \big\| \jp{\ga_T u} \big\|_{\Hdiv}^2 \\
	& \leq \ C \Big( \|\curl u\|_{L^2(\R^3\setminus\Ga)^3}^2 + \|u\|_{L^2(\R^3\setminus\Ga)^3}^2 \Big)\\
	&=  \ C |s|^2 \Big( \|s\inv \curl u\|_{L^2(\R^3\setminus\Ga)^3}^2 + |s|^{-2} \|u\|_{L^2(\R^3\setminus\Ga)^3}^2 \Big)\\
	& \leq \ C |s|^2 \max\{1,|s|^{-2}(\eps\mu)\inv\} \Big(\|s\inv \curl u\|_{L^2(\R^3\setminus\Ga)^3}^2 + \eps\mu \|u\|_{L^2(\R^3\setminus\Ga)^3}^2 \Big) ,
	\end{align*}
	and for $\ga_N$ we use the fact that $\ga_N u = \gamma_T (s\inv \curl u)$:
	\begin{align*}
	&(\eps\mu)\inv \|\vphi\|_{\Hdiv}^2 = (\eps\mu)\inv \big\| \jp{\ga_T (s\inv \curl u)} \big\|_{\Hdiv}^2 \\
	&\leq \ C (\eps\mu)\inv \Big(\|s\inv \curl \curl u\|_{L^2(\R^3\setminus\Ga)^3}^2 + \|s\inv \curl u\|_{L^2(\R^3\setminus\Ga)^3}^2 \Big) \\
	&=\ C \Big( \eps\mu \|s u\|_{L^2(\R^3\setminus\Ga)^3}^2  + (\eps\mu)\inv \|s\inv \curl u\|_{L^2(\R^3\setminus\Ga)^3}^2 \Big) \\
	&\leq \ C |s|^2 \max\{1 ,|s|^{-2}(\eps\mu)\inv\} \Big(\|s\inv \curl u\|_{L^2(\R^3\setminus\Ga)^3}^2 + \eps\mu\|u\|_{L^2(\R^3\setminus\Ga)^3}^2 \Big)
	\end{align*}
	where, for the first inequalities in both estimates, we used the trace inequality of Lemma~\ref{lemma: trace ineq}. Extraction of factors and dividing through completes the proof.
	\QED\end{proof}

\section{Boundary integral formulation of Maxwell's equations}
\label{section: boundary int form}

Let us consider the first order formulation of Maxwell's equations, in the following form:
\begin{equation*}
\begin{aligned}
\eps \pa_t \E - \curl \H =&\ \J \\
\mu \pa_t \H + \curl \E =&\ 0 \\
\div\big(\eps \E\big) =&\ 0 \\
\div\big(\mu \H\big) =&\ 0
\end{aligned}
\qquad\qquad \textrm{ in } \Om,
\end{equation*}
with appropriate initial and boundary conditions. If the initial conditions satisfy the last two equations, then they hold for all times, see \cite{Monk,ChenMonk}, therefore these conditions are assumed to hold.
The permeability and permittivity is denoted by $\mu$ and $\eps$, respectively, and they are assumed to be positive constants, while $\J$ denotes the electric current density.

Using the relation $\pa_t\H = -\mu\inv \curl \E$, the above equation can be written as the second order problem
\begin{equation*}
\eps\mu \pa_t^2 \E + \curl \curl \E = \dot\J \qin \Om ,
\end{equation*}
with $\dot\J = \pa_t \J$.

Setting $\dot\J=0$, applying Laplace transformation, and writing $u$ instead of ${\cal L}\E$, we obtain the time-harmonic version \eqref{eq: time-harmonic Maxwell }.

\subsection{Recap: Temporal convolutions and Herglotz theorem}
\label{subsec:herglotz}

We  recall an operator-valued continuous-time Herglotz theorem from \cite[Section 2.2]{leapfrog}, which is crucial for transferring the coercivity result of Lemma~\ref{lemma: coercivity} from the Maxwell's equation in the Laplace domain to the time-dependent Maxwell's equation. We describe the result in an abstract Hilbert space setting. 

Let $\calH$ be a complex Hilbert space, with dual $\calH'$ and anti-duality $\langle\cdot,\cdot\rangle$. Let $B(s):\calH\to \calH'$ and $R(s):\calH \to \calH$ be both analytic families of bounded linear operators for $\re s \geq \sigma>0$, satisfying the uniform bounds:
\begin{equation*}
\|B(s)\|_{\calH'\leftarrow \calH} \leq M|s|^\mu \andquad \|R(s)\|_{\calH'\leftarrow \calH} \leq M|s|^\mu, \qquad \re s \geq \sigma .
\end{equation*}
For any integer $m > \mu+1$, we define the integral kernel
\begin{equation*}
B_m(t) = \frac{1}{2\pi\iu} \int_{\sigma+\iu\R} e^{st}s^{-m}B(s) \d s .
\end{equation*}
For a function $w\in C^m([0,T],\calH)$ with vanishing initial data, $w(0)=w'(0)=\dotsb=w^{(m-1)}(0)=0$, we let
\begin{equation*}
(B(\pa_t)w)(t) = \Big(\diff\Big)^m \int_0^t B_m(t-\tau)w(\tau) \d \tau ,
\end{equation*}
that is, $B(\pa_t)w$ is the distributional convolution of the inverse Laplace transform of $B(s)$ with $w$.

The operator-valued version of the classical Herglotz theorem from \cite{leapfrog} yields the following result.
\begin{lemma}[\cite{leapfrog}, Lemma~2.2]
	\label{lemma: Herglotz}
	In the above setting, the following two statements are equivalent:
	\begin{enumerate}
		\item[(i)] $\re \langle w,B(s)w \rangle \geq \beta \|R(s)w\|^2$, for any $w \in \calH$, $\re s \geq \sigma$;\\[-2mm]
		\item[(ii)] $\int_0^\infty e^{-2\sigma t} \re \langle w(t),B(\pa_t)w(t) \rangle \d t \geq \beta \int_0^\infty e^{-2\sigma t} \|R(\pa_t)w(t)\|^2 \d t$, for all $w\in C^m([0,T],\calH)$, with $w(0)=w'(0)=\dotsb=w^{(m-1)}(0)=0$, and for all $t\geq0$.
	\end{enumerate}
\end{lemma}

\subsection{Calderon operator for Maxwell's equations}
\label{subsection: Calderon op for Maxwell's eqn}

Consider the second order formulation of Maxwell's equations in three dimensions:
\begin{align*}
\eps\mu \pa_t^2 \E + \curl \curl \E =&\ \dot\J & \qin & \R^3 \times [0,T], \\
\E(x,0) =&\ \E_0 & \qin & \R^3, \\
\pa_t \E(x,0) =&\ \H_0 & \qin & \R^3.
\end{align*}
Let $\Om\subset\R^3$ be a bounded Lipschitz domain, with boundary $\Ga$, and further assume that the initial values and $\dot\J$ are supported within $\Om$.

We rewrite this problem as an interior problem over $\Om$:
\begin{align*}
\eps\mu \pa_t^2 \E^- + \curl \curl \E^- =&\ \dot\J & \qin & \Om \times [0,T], \\
\E^-(x,0) =&\ \E_0 & \qin & \Om, \\
\pa_t \E^-(x,0) =&\ \H_0 & \qin & \Om,
\end{align*}
and as an exterior problem over $\Om^+ = \R^3 \setminus \overline{\Om}$:
\begin{align*}
\eps\mu \pa_t^2 \E^+ + \curl \curl \E^+ =&\ 0 & \qin & \Om^+ \times [0,T], \\
\E^+(x,0) =&\ 0 & \qin & \Om^+, \\
\pa_t \E^+(x,0) =&\ 0 & \qin & \Om^+.
\end{align*}
The two problems are \emph{coupled} by the transmission conditions:
\begin{equation*}
\ga_T^- \E^- = \ga_T^+ \E^+ \andquad \ga_N^- \E^- = \ga_N^+ \E^+ .
\end{equation*}

Using the temporal convolution operators of Section~\ref{subsec:herglotz}, the solution of the exterior problem is given as
\begin{equation*}
\E^+ = \calS(\pa_t)\vphi + \calD(\pa_t)\psi,
\end{equation*}
with boundary densities
\begin{equation*}
\vphi = -\ga_N^+ \E^+ \andquad \psi = -\ga_T^+ \E^+,
\end{equation*}
which satisfy the equation
\begin{equation*}
B(\pa_t)\vect{\vphi}{\psi} = \frac{ \C }{2} \vect{\ga_T^- \E^-}{-\ga_N^- \E^-} .
\end{equation*}
Here $B(\pa_t)$ is the temporal convolution operator with the distribution whose Laplace transform is the Calderon operator $B(s)$ defined in \eqref{def-B}.

\subsection{First order formulation}

From now on, we use Maxwell's equations in their first order formulation 
on the interior domain $\Om$ (and we omit the omnipresent superscript $^-$):
\begin{equation}
\label{eq: first order Maxwell}
\begin{aligned}
\eps \pa_t\E =&\ \curl \H + \J \\
\mu \pa_t\H =&\ -\curl \E
\end{aligned}
\qin \Om \times [0,T] ,
\end{equation}
with the coupling through the Calderon operator as
\begin{equation*}
B(\pa_t) \vpp
= \frac{ \C }{2} \vect{\ga_T \E}{-\ga_N \E} ,
\end{equation*}
where $\vphi=-\ga_N \E$ and $\psi=-\ga_T \E$. In addition, by $-\mu\pa_t\H = \curl \E$ we obtain
\begin{equation}
\label{eq: boundary term equalities}
\begin{aligned}
\vphi =&\ -\ga_N \E =  -\ga_T (\pa_t\inv \curl \E)
= \mu \ga_T \H\\
\psi =&\ -\ga_T \E ,
\end{aligned}
\end{equation}
where we also used \eqref{eq: Laplace domain boundary term equalities}. Hence, $-\ga_N \E = \mu \ga_T \H$ and

\begin{equation*}
B(\pa_t) \vpp
= \frac{ \C }{2} \vect{\ga_T \E}{-\ga_N \E}
= \Half \vect{\mu\inv \ga_T \E}{\ga_T \H} .
\end{equation*}

\subsection{Coercivity of the time-dependent Calderon operator}

In the same way as in \cite[Lemma 4.1]{leapfrog} for the acoustic wave equation, the coercivity of the Calderon operator $B(s)$ for the time-harmonic Maxwell's equation as given by Lemma~\ref{lemma: coercivity} together with the operator-valued continuous-time Herglotz theorem as stated in Lemma 2.3 yields coercivity of the time-dependent Calderon operator $B(\partial_t)$.

\begin{lemma}
	\label{lemma: time-cont coercivity}
	With the constant $\beta>0$ from Lemma~\ref{lemma: coercivity} we have that
	\begin{align*}
	&\ \int_0^T e^{-2t/T} \biggl\la \vect{\vphi(\cdot,t)}{\psi(\cdot,t)}, B(\pa_t) \vect{\vphi(\cdot,t)}{\psi(\cdot,t)} \biggr\ra_\TG \d t \\
	&\ \geq \beta c_T \int_0^T e^{-2t/T} \Big( (\eps\mu)\inv \|\pa_t\inv \vphi(\cdot,t)\|_{\Hdiv}^2 + \|\pa_t\inv \psi(\cdot,t)\|_{\Hdiv}^2 \Big) \d t
	\end{align*}
	for arbitrary $T>0$ and for all $\vphi \in C^4([0,T],\Hdiv)$ and all $\psi \in C^4([0,T],\Hdiv)$ with $\vphi(\cdot,0)=\pa_t\vphi(\cdot,0)=\pa_t^2\vphi(\cdot,0)=\pa_t^3\vphi(\cdot,0)=0$ and $\psi(\cdot,0)=\pa_t\psi(\cdot,0)=\pa_t^2\psi(\cdot,0)=\pa_t^3\psi(\cdot,0)=0$, and with constant $c_T = \m{T\inv}$. 
\end{lemma}
%
%

A Gronwall argument then yields the following energy estimate; see \cite[Lemma~4.2]{leapfrog}.
\begin{lemma}
	\label{lemma: general energy-like est}
	Let the functions $\,\calE:[0,T] \to [0,\infty)$, $\calF:[0,T] \to \R$, and $\vphi,\psi \in C^2([0,T], \Hdiv)$ with $\vphi(\cdot,0)=\pa_t\vphi(\cdot,0)=\pa_t^2\vphi(\cdot,0)=\pa_t^3\vphi(\cdot,0)=0$, $\psi(\cdot,0)=\pa_t\psi(\cdot,0)=\pa_t^2\psi(\cdot,0)=\pa_t^3\psi(\cdot,0)=0$, be such that for all $t\in[0,T]$
	\begin{equation*}
	\dot \calE(t) + \biggl\la \vect{\vphi(\cdot,t)}{\psi(\cdot,t)}, B(\pa_t) \vect{\vphi(\cdot,t)}{\psi(\cdot,t)} \biggr\ra_\TG = \calF(t).
	\end{equation*}
	Then, with $c_T = \m{T\inv}$,
	\begin{equation}
	\label{eq: energy estimate}
	\begin{aligned}
	\calE(T) + &\ \beta c_T \int_0^T \e^{-2t/T} \Big( (\eps\mu)\inv \|\pa_t\inv \vphi(\cdot,t)\|_{\Hdiv}^2 + \|\pa_t\inv \psi(\cdot,t)\|_{\Hdiv}^2 \Big) \d t
	\\
	&\ \leq \e^2\,\calE(0) + \int_0^T \e^{2(1-t/T)} \calF(t) \d t.
	\end{aligned}
	\end{equation}
\end{lemma}

\subsection{Weak formulation and energy estimate}
Analogously to  \cite{abboud2011coupling,leapfrog}, a symmetric weak form of \eqref{eq: first order Maxwell} is obtained on using
\begin{equation*}
(\curl u , v) = \half (\curl u , v) + \half (u , \curl v) - \half \la \ga_T u , \ga_T v \ra_\TG ,
\end{equation*}
and using \eqref{eq: boundary term equalities} for the boundary term. Here $(\cdot,\cdot)$ denotes the standard $L^2(\Om)^3$ inner product.

The coupled weak problem then reads: find $\E, \H \in H(\curl,\Om)$ and $\vphi, \psi \in \Hdiv$ such that
\begin{equation}
\label{eq: weak formulation}
\begin{aligned}
&\ (\eps \pa_t\E,w) = \half (\curl \H , w) + \half (\H , \curl w)
- \half \la \ga_T \H , \ga_T w \ra_\TG + (\J,w) \\
&\ \hphantom{(\eps \pa_t\E,w)} = \half (\curl \H , w) + \half (\H , \curl w)
- \half \la \mu\inv \vphi , \ga_T w \ra_\TG + (\J,w) , \\
&\ (\mu \pa_t\H,z) = -\half (\curl \E , z) - \half (\E , \curl z)
+ \half \la \ga_T \E , \ga_T z \ra_\TG \\
&\ \hphantom{(\mu \pa_t\H,z)} = -\half (\curl \E , z) - \half (\E , \curl z)
- \half \la \psi , \ga_T z \ra_\TG  ,\\
&\ \biggl\la \vect{\xi}{\eta}, B(\pa_t) \vpp \biggr\ra_\TG = \half \Big( \la\xi,\mu\inv \ga_T \E\ra_\TG + \la \eta,\ga_T \H \ra_\TG \Big)
\end{aligned}
\end{equation}
hold for arbitrary $w,z \in H(\curl,\Om)$, and $\xi, \eta \in \Hdiv$.

While this weak formulation is apparently non-standard for Maxwell's equations, we will see that it is extremely useful, in the same way as the analogous formulation proved to be for the acoustic case in \cite{abboud2011coupling,leapfrog}.

Testing with $w=\E$, $z=\H$ and $\xi = \vphi$, $\eta=\psi$ in \eqref{eq: weak formulation}, by using \eqref{eq: boundary term equalities} we obtain
\begin{equation*}
\begin{aligned}
&\ (\eps \pa_t\E,\E) = \half (\curl \H , \E) + \half (\H , \curl \E)
- \half \la \mu\inv \vphi , \ga_T \E \ra_\TG + (\J,\E) , \\
&\ (\mu \pa_t\H,\H) = -\half (\curl \E , \H) - \half (\E , \curl \H)
- \half \la \psi , \ga_T \H \ra_\TG  ,\\
&\ \biggl\la \vect{\vphi}{\psi}, B(\pa_t) \vpp \biggr\ra_\TG = \half \Big( \la\vphi,\mu\inv \ga_T \E\ra_\TG + \la \psi,\ga_T \H \ra_\TG \Big) ,
\end{aligned}
\end{equation*}
and summing up the three equations yield
\begin{equation*}
\diff \Big(\frac\eps2 \|\E\|_{L^2(\Om)^3}^2 +  \frac\mu2 \|\H\|_{L^2(\Om)^3}^2 \Big) + \biggl\la \vpp , B(\pa_t)\vpp \biggr\ra_\TG = (\J,\E)  .
\end{equation*}
For $\J=0$,  the coercivity of the continuous-time Calderon operator, as stated in Lemmas~\ref{lemma: time-cont coercivity} and~\ref{lemma: general energy-like est}, yields that the electromagnetic energy
\begin{equation*}
\calE(t) = \frac\eps2 \|\E(\cdot,t)\|_{L^2(\Om)^3}^2 + \frac\mu2 \|\H(\cdot,t)\|_{L^2(\Om)^3}^2 ,
\end{equation*}
satisfies the energy estimate \eqref{eq: energy estimate} (with $\calF=0$) for arbitrary $T>0$.

\section{Discretization}
\label{section: discretization}

\subsection{Space discretization: dG and BEM}
\label{subsection: dG}

For the spatial discretization we use, as an example, the central flux discontinuous Galerkin (dG) discretization from \cite{HochbruckSturm} (see also \cite{DiPietroErn,HesthavenWarburton}) in the interior and continuous linear boundary elements on the surface.

We triangulate the bounded polyhedral domain $\Om$ by simplicial triangulations $\calT_h$, where $h$ denotes the maximal element diameter.  For our theoretical results we consider a quasi-uniform and
contact-regular family of such triangulations with $h\to0$; see e.g.~\cite{DiPietroErn} for these notions.
We adopt the following notation from \cite[Section~2.3]{HochbruckSturm}: The faces $\calF_h$ of $\calT_h$, decomposed into boundary and interior faces: $\calF_h = \Fb \cup \Fi$. The normal of an interior face $F\in\Fi$ is denoted by $\nu_F$. It is kept fixed and is the outward normal of one of the two neighbouring mesh elements. We denote by $K_F$ that neighbouring element into which $\nu_F$ is directed. The outer faces of $\calT_h$ are used as the triangulation of the boundary $\Gamma$.

The dG space of vector valued functions, which are elementwise linear in each component, is defined as
\begin{equation*}
V_h = \big\{ v_h \in L^2(\Om) \,;\, \ v_h|_K \textnormal{ is at most linear, for all } K \in \calT_h \big\}^3 \not\subset H(\curl,\Om).
\end{equation*}
The boundary element space $\Psi_h$ is taken as
\begin{equation*}
\Psi_h = \bigl\{ \chi_h \times \nu \,;\ \chi_h:\Gamma\to \R^3 \textnormal{ is piecewise linear and continuous}\bigr\} \subset \Hdiv.
\end{equation*}
The corresponding nodal basis functions are denoted by $(b_j^\Omega)$ and $(b_k^\Ga)$, respectively.
%
%
Jumps and averages over faces $F \in \Fi$ are denoted analogously as for trace operators on $\Ga$, see Section~\ref{subsection: boundary integral op}:
\begin{equation*}
\jp{w}_F = \ga_F^-w - \ga_F^+w \andquad \av{w}_F = \half(\ga_F^-w + \ga_F^+w),
\end{equation*}
where $\ga_F$ is the usual trace onto the face $F$. We often omit the subscript as it will always be clear from the context.

The discrete $\curl$ operator with centered fluxes was presented in \cite[Section~2.3]{HochbruckSturm}:
\begin{equation*}
\begin{aligned}
(\curl_h u_h , w_h) =&\ \sum_{K\in\calT_h} (\curl u_h , w_h)_K + \sum_{F\in\Fi} - \la \jp{u_h} , \av{w_h} \ra_{F} .
\end{aligned}
\end{equation*}
By the arguments of the proof of Lemma~2.2 in \cite{HochbruckSturm}, we obtain that the discrete curl operator satisfies the discrete version of Green's formula \eqref{eq: Green},
\begin{equation}
\label{eq: Green-h}
(\curl_h u_h , w_h) - ( u_h , \curl_h w_h) = \la \gamma_T u_h, \gamma_T w_h \ra_\Ga.
\end{equation}
%
The $\curl_h$ operator is well defined on $H(\curl,\Om)\cap H^1(\calT_h)^3$, with the broken Sobolev space
\begin{equation*}
H^k(\calT_h) =\big\{ v\in L^2(\Om) \,:\, \ v|_K \in H^k(K) \textnormal{ for all } K \in \calT_h  \big\} \qquad (k\in\N),
\end{equation*}
which is a Hilbert space with natural norm and seminorm $\|v_h\|_{H^k(\calT_h)}$ and $|v_h|_{H^k(\calT_h)}$, respectively.

Using the above  discrete $\curl$ operator, the semidiscrete problem reads as follows: Find $\E_h, \H_h \in V_h$ and $\vphi_h, \psi_h \in \Psi_h$ such that for all $w_h,z_h \in V_h$ and $\xi_h, \eta_h \in \Psi_h$,
\label{eq: semidiscrete problem}
\begin{align}
\nonumber
&\ (\eps \pa_t\E_h,w_h) = \half (\curl_h \H_h , w_h) + \half (\H_h , \curl_h w_h) - \half \la \mu\inv \vphi_h , \ga_T w_h \ra_\TG + (\J,w_h) , \\[1mm]
\label{eq:dg-bem}
&\ (\mu \pa_t\H_h,z_h) = -\half (\curl_h \E_h , z_h) - \half (\E_h , \curl_h z_h) - \half \la \psi_h , \ga_T z_h \ra_\TG  ,\\[2mm]
\nonumber
&\ \biggl\la \vect{\xi_h}{\eta_h}, B(\pa_t) \vect{\vphi_h}{\psi_h} \biggr\ra_\TG = \half \Big( \la\xi_h,\mu\inv \ga_T \E_h\ra_\TG + \la \eta_h,\ga_T \H_h \ra_\TG \Big).
\end{align}
All expressions are to be interpreted in a piecewise sense if necessary.

We collect the nodal values of the semidiscrete electric and magnetic field into the vectors $\bfE, \bfH$, and similarly the nodal vectors of the boundary densities are denoted by $\bfphi$ and $\bfpsi$. Upright boldface capitals always denote matrices of the discretization.

We obtain the following coupled system of ordinary differential equations and integral equations for the nodal values:
\begin{equation}
\label{eq: matrix formulation}
\begin{aligned}
&\ \eps\bfM \dot\bfE = -\bfD\bfH - \bfC_0 \bfphi + \bfM\bfJ , \\
&\ \mu\bfM \dot\bfH =  \bfD^T\bfE -\bfC_1\bfpsi  ,\\
&\ \bfB(\pa_t) \vect{\bfphi}{\bfpsi} = \vect{\bfC_0^T \bfE}{\bfC_1^T \bfH} .
\end{aligned}
\end{equation}
The matrix $\bfM$ denotes the symmetric positive definite mass matrix, while the other matrices are defined as
\begin{equation*}
\bfD|_{jj'} = -\half (\curl_h b_{j'}^\Omega , b_j^\Omega) - \half (b_{j'}^\Omega , \curl_h b_j^\Omega) ,
\end{equation*}
which happens to be a symmetric matrix, and
\begin{equation*}
\bfC_1|_{jk} = \half \la b_k^\Ga , \ga_T b_j^\Omega \ra_\TG ,  \qquad \bfC_0=\mu\inv \bfC_1 .
\end{equation*}

The matrix $\bfB(s)$ is given by
\begin{equation*} 
\bfB(s) = \mu\inv
\left(
\begin{array}{cc}
\bfV(s) & \bfK(s) \\
-\bfK(s) & -\bfV(s) \\
\end{array}
\right) ,
\end{equation*}
where the blocks have entries
\begin{equation*}
\bfV(s)|_{kk'} = \half \la b_{k'}^\Ga , V(s) b_k^\Ga \ra_\TG \andquad
\bfK(s)|_{kk'} = \half \la b_{k'}^\Ga , K(s) b_k^\Ga \ra_\TG .
\end{equation*}
For this matrix we have the following coercivity estimate.

\begin{lemma}
	\label{lemma: coercivity-discrete}
	With $\beta>0$ and $ \m{s}$ from Lemma~\ref{lemma: coercivity},  the matrix $\bfB(s)$ satisfies
	\begin{align*}
	\re  \vect{\bfphi}{\bfpsi}^* \bfB(s) \vect{\bfphi}{\bfpsi} 
	\geq \beta \,\m{s}  \Big( (\eps\mu)\inv (s\inv\bfphi)^* \bfM_\Ga  (s\inv\bfphi) + (s\inv\bfpsi)^* \bfM_\Ga  (s\inv\bfpsi) \Big)
	\end{align*}
	for $\re s>0$ and for all $\bfphi, \bfpsi \in \mathbb{C}^{N_{\Ga}}$, where the mass matrix $\bfM_\Ga$, for the inner product $(\cdot,\cdot)_\Ga$ corresponding to the norm \eqref{Hdiv-norm} on $\Hdiv$, is
	defined by $\bfM_\Ga |_{kk'} = (b_{k'}^\Ga, b_{k}^\Ga)_\Ga$.
\end{lemma}

\begin{proof} The result follows from Lemma~\ref{lemma: coercivity} on noting that for the vectors
	$\bfphi = (\vphi_k)$ and $\bfpsi = (\psi_k)$ and the corresponding boundary functions in $\Psi_h\subset\Hdiv$
	$$
	\vphi_h = \sum_{k=1}^{N_\Ga} \vphi_k b_k^\Ga \quad\hbox{ and } \quad
	\psi_h = \sum_{k=1}^{N_\Ga} \psi_k b_k^\Ga,
	$$
	we have
	$$
	\vect{\bfphi}{\bfpsi}^* \bfB(s) \vect{\bfphi}{\bfpsi}  = \biggl\la \vect{\vphi_h}{\psi_h}, B(s) \vect{\vphi_h}{\psi_h} \biggr\ra_\TG
	$$
	and $\| \vphi_h \|_{\Hdiv} ^2 = \bfphi^* \bfM_\Ga \bfphi$.
	\qed
\end{proof}

Let us emphasize the following observation:
{\it The above matrix--vector formulation \eqref{eq: matrix formulation} is formally the same as the one for the acoustic wave equation in \cite[Section~5.1]{leapfrog}, with the same coercivity estimate for the boundary operator $\bfB(\pa_t)$ by Lemmas~\ref{lemma: coercivity-discrete} and~\ref{lemma: time-cont coercivity}.
	As an important consequence, the stability  results proven in \cite{leapfrog} hold for the present case as well.}

\begin{remark}
	The choice of a dG method in the interior and of continuous boundary elements for the spatial discretizations is not necessary for our analysis. Other space discretization methods, for instance the ones going back to Raviart and Thomas \cite{RaviartThomas}, N\'{e}d\'{e}lec \cite{Nedelec-mixed}, and many others, detailed in the excellent survey article \cite{Hiptmair-acta}, or locally divergence-free methods such as \cite{BrennerLiSung,CockburnLiShu}, could also be used as long as they yield a 
	matrix--vector formulation of the form \eqref{eq: matrix formulation} and a coercivity estimate as in Lemma~\ref{lemma: coercivity-discrete}.
\end{remark}


\subsection{Recap: Convolution quadrature}
\label{subsection: cq}

Following \cite[Section~2.3]{leapfrog} we give a short recap of convolution quadrature and  introduce some notation. For more details see  \cite{LubichCQ,Lubich-multistep,Lubich-revisited} and \cite{Banjai}.

Convolution quadrature (CQ) discretizes the convolution $B(\pa_t)w(t)$ by the discrete convolution
\begin{equation*}
(B(\pa_t^{\dt})w)(n\dt) = \sum_{j=0}^n B_{n-j} w(j\dt) ,
\end{equation*}
where the weights $B_n$ are defined as the coefficients of
\begin{equation*}
B\Big( \frac{\delta(\zeta)}{\dt} \Big) = \sum_{n=0}^\infty B_n \zeta^n .
\end{equation*}
In the present paper we  choose
\begin{equation*}
\delta(\zeta) = (1-\zeta) + \half(1-\zeta)^2 ,
\end{equation*}
which corresponds to the second-order backward difference formula.

From \cite{Lubich-multistep}, it is known that the method is of order two,
\begin{equation*}
\| (B(\pa_t)w)(t) - (B(\pa_t^{\dt})w)(t)\| = O(\dt^2), \textnormal{ uniformly in } t=n\dt\leq T,
\end{equation*}
for functions $w$ that are sufficiently smooth including their extension by $0$ to negative values of $t$.
An important property of this discretization is that it preserves the coercivity of the continuous-time convolution in the time discretization. We have the following result.

\begin{lemma} [\cite{leapfrog}, Lemma~2.3]
	\label{lemma: CQ coercivity}
	In the setting of Lemma~\ref{lemma: Herglotz} condition (i) implies, for $\sigma\dt >0$ small enough and with $ \rho = e^{-\sigma\dt} + O(\dt^2)$,
	\begin{equation*}
	\sum_{n=0}^\infty \rho^{2n} \re \langle w(n\dt),B(\pa_t^{\dt})w(n\dt)\rangle \geq \gamma \sum_{n=0}^\infty \rho^{2n} \|R(\pa_t^{\dt})w(n\dt)\|^2,
	\end{equation*}
	for any function $w:[0,\infty) \to V$ with finite support.
\end{lemma}

\subsection{Coercivity of the time-discretized Calderon operator}

Combining Lemma~\ref{lemma: coercivity} and  Lemma~\ref{lemma: CQ coercivity} yields the following coercivity property of the CQ time-discretization of the time-dependent Calderon operator considered in Lemma~\ref{lemma: time-cont coercivity}.

\begin{lemma}
	\label{lemma: time-discrete coercivity} In the situation of Lemma~\ref{lemma: time-cont coercivity}, we have for $N\dt =T$ and $0<\dt\le\dt_0$  that
	\begin{align*}
	&\ \dt\sum_{n=0}^N e^{-2t_n/T} \biggl\la \vect{\vphi(\cdot,t_n)}{\psi(\cdot,t_n)}, B(\pa_t^{\dt}) \vect{\vphi}{\psi} (\cdot,t_n)\biggr\ra_\TG  \\
	&\ \geq \beta c_T  \dt\sum_{n=0}^N e^{-2t_n/T} \Big( (\eps\mu)\inv \|(\pa_t^\dt)\inv \vphi(\cdot,t_n)\|_{\Hdiv}^2 + \|(\pa_t^\dt)\inv \psi(\cdot,t_n)\|_{\Hdiv}^2 \Big)
	\end{align*}
	for all sequences $(\vphi(\cdot,t_n))_{n=0}^N$ and  $(\psi(\cdot,t_n))_{n=0}^N$ in $\Hdiv$, with $c_T = c\,\m{T\inv}$ for a $c>0$ (which depends only on $\dt_0$ and tends to 1 as $\dt_0$ goes to zero). 
\end{lemma}

\subsection{Time discretization: leapfrog and CQ}
Similarly to \cite{leapfrog}, we use the leapfrog or St\"{o}rmer--Verlet scheme (see, e.g., \cite{HairerLubichWanner}) in the interior:
\begin{equation}
\label{eq: full discr - a}
\begin{aligned}
\mu \bfM \bfH^{n+1/2} =&\ \mu \bfM\bfH^{n} + \half\dt \bfD\bfE^{n} -\half\dt \bfC_1\bfpsi^{n} , \\
\eps \bfM \bfE^{n+1} =&\ \eps \bfM \bfE^{n} - \dt \bfD^T\bfH^{n+1/2} -\dt\bfC_0\bfphi^{n+1/2} + \dt\bfM\bfJ^{n+1/2} , \\
\mu \bfM \bfH^{n+1} =&\ \mu \bfM\bfH^{n+1/2} + \half\dt \bfD\bfE^{n+1} - \half\dt \bfC_1\bfpsi^{n+1} ,
\end{aligned}
\end{equation}
where the last substep of the previous step and the first substep  can be combined to a step from $\bfH^{n-1/2}$ to $\bfH^{n+1/2}$ when no output at $t_{n}$ is needed:
$$
\mu \bfM \bfH^{n+1/2} = \mu \bfM\bfH^{n-1/2} + \dt \bfD\bfE^{n} -\dt \bfC_1\bfpsi^{n} .
$$
This is coupled
with convolution quadrature on the boundary
\begin{equation}
\label{eq: full discr - b}
\bigg[
\bfB(\pa_t^{\dt}) \vect{\bfphi}{\bar\bfpsi}\bigg]^{n+1/2}
= \vect{\bfC_0^T \bar\bfE^{n+1/2}}{\bfC_1^T \bfH^{n+1/2}}
+ \vect{0}{- \alpha \dt^2 \mu\inv\bfC_1^T \bfM\inv \bfC_1 \dot\bfpsi^{n+1/2}},
\end{equation}
where the operation $\bar f^{n+1/2} = \half(f^{n+1} + f^{n})$ is averaging in time and $\dot\bfpsi^{n+1/2} = (\bfpsi^{n+1} - \bfpsi^{n})/\dt$. The second term on the right-hand side is a stabilizing term, with a parameter $\alpha>0$. The role of this extra term becomes clear from the proof of the stability result for the acoustic wave equation \cite[Lemma~8.1]{leapfrog}, which applies to the Maxwell case as well.

Like for the acoustic case, the choice $\alpha=1$ yields a stable scheme under the CFL condition $\dt \|\bfM^{-1/2}\bfD\bfM^{-1/2}\|_2 \leq \sqrt{\eps\mu}$. Up to a factor 2 this is  the CFL condition for the leapfrog scheme for the equation with natural boundary conditions.

In each time step, a linear system with the matrix  $\bfB_0 + \dt \bfG$ needs to be solved for $\bfphi^{n+1/2}$ and $\bar\bfpsi^{n+1/2}$, where $\bfB_0=\bfB(\delta(0)/\dt)$ and
\begin{align*}
\bfG=&\ \left(
\begin{array}{cc}
\half \eps\inv \bfC_0^T \bfM\inv \bfC_0 & 0 \\
0 & 2 \alpha \mu\inv\bfC_1^T \bfM\inv \bfC_1 \\
\end{array}
\right) .
\end{align*}
By the coercivity Lemma~\ref{lemma: coercivity}, $\bfB_0 + \bfB_0^T$ is positive definite. Moreover, $\bfG$ is symmetric positive definite.

\section{Stability results and error bounds for the spatial semidiscretization}
\label{section: semidiscrete results}

Using that the obtained discrete system \eqref{eq: matrix formulation} is of the same form and with the same coercivity property as for the acoustic wave equation, the stability results carry over from Section~6 of \cite{leapfrog}. Only minor technical modifications are needed, such as using the appropriate energy and norms. The only point where the analysis of the semidiscrete problem deviates from the acoustic case is the consistency error estimates, which require special care.

\subsection{Stability}

We consider a system with additional inhomogeneities $j_h,g_h:[0,T]\to L^2(\Omega)^3$ and $\rho_h,\sigma_h:[0,T]\to \Hdiv$, which will later be obtained as the system of error equations with the defects of an interpolation of the exact solution. The coupled system
\begin{align}
\nonumber
&\ (\eps \pa_t\E_h,w_h) = \half (\curl_h \H_h , w_h) + \half (\H_h , \curl_h w_h) - \half \la \mu\inv \vphi_h , \ga_T w_h \ra_\TG + (j_h,w_h) , \\[1mm]
\label{eq:dg-bem-stab}
&\ (\mu \pa_t\H_h,z_h) = -\half (\curl_h \E_h , z_h) - \half (\E_h , \curl_h z_h) - \half \la \psi_h , \ga_T z_h \ra_\TG +(g_h,w_h) ,\\[2mm]
\nonumber
&\ \biggl\la \vect{\xi_h}{\eta_h}, B(\pa_t) \vect{\vphi_h}{\psi_h} \biggr\ra_\TG = \half \Big( \la\xi_h,\mu\inv \ga_T \E_h\ra_\TG + \la \eta_h,\ga_T \H_h \ra_\TG \Big) \\
&\ \hphantom{\biggl\la \vect{\xi_h}{\eta_h}, B(\pa_t) \vect{\vphi_h}{\psi_h} \biggr\ra_\TG = } + (\xi_h,\rho_h)_\Ga +(\eta_h,\sigma_h)_\Ga
\end{align}
where $(\cdot,\cdot)_\Ga$ denotes the inner product on $\Hdiv$, has the matrix-vector formulation
\begin{equation}
\label{eq: residuals - matrix formulation}
\begin{aligned}
&\ \eps \bfM\dot\bfE = -\bfD\bfH -\bfC_0\bfphi + \bfM\bfj , \\
&\ \mu \bfM\dot\bfH =  \bfD^T\bfE -\bfC_1\bfpsi + \bfM\bfg ,\\
&\ \bfB(\pa_t) \vect{\bfphi}{\bfpsi} = \vect{\bfC_0^T \bfE}{\bfC_1^T \bfH} + \vect{\bfM_\Ga\bfrho}{\bfM_\Ga\bfsigma},
\end{aligned}
\end{equation}
where $\bfM_\Ga$ is the boundary mass matrix  with entries $\bfM_\Ga\vert_{k',k} = ( b_{k'}^\Ga, b_k^\Ga )_\Ga$.
The solution of this system can be bounded in terms of $\bfj, \bfg, \bfrho, \bfsigma$ by the stability results proven in Lemma~6.1--6.3 in \cite{leapfrog}.

We immediately translate the stability lemmas of \cite{leapfrog} into the functional analytic setting. The energy estimate of Lemma~6.1 of \cite{leapfrog} becomes the following.

\begin{lemma}
	\label{lemma: semidiscr stability - energy estimate}
	The semidiscrete energy
	\begin{equation*}
	\calE_h(t)=\half \Big( \eps\|\E_h(\cdot,t)\|_{L^2(\Omega)^3}^2 + \mu\|\H_h(\cdot,t)\|_{L^2(\Omega)^3}^2 \Big),
	\end{equation*}
	satisfies the bound, for $t>0$,
	\begin{align*}
	&\ \calE_h(t) \leq C(\beta)\biggl( \calE_h(0) + t \int_0^t
	\bigl(\|j_h(\cdot,\tau)\|_{L^2(\Omega)^3}^2 + \|g_h(\cdot,\tau)\|_{L^2(\Omega)^3}^2\bigr) \d \tau \\
	&\ \qquad + \max\{t^2,t^6 (\eps\mu)^2\} \int_0^t
	\bigl(\|\pa_t^2 \rho_h(\cdot,\tau)\|_{\Hdiv}^2 + \|\pa_t^2 \sigma_h(\cdot,\tau)\|_{\Hdiv}^2 \bigr)
	\d \tau \biggr),
	\end{align*}
	provided that $\rho_h(\cdot,0)=\pa_t\rho_h(\cdot,0)=0$ and $\sigma_h(\cdot,0)=\pa_t\sigma_h(\cdot,0)=0$.
\end{lemma}

The estimates for the boundary functions of Lemma 6.3 of \cite{leapfrog} now translate into the following.

\begin{lemma}
	\label{lemma: semidiscr stability - boundary functions}
	For $t>0$, the boundary functions are bounded as
	\begin{align*}
	&\ \int_0^t \bigl(\|\vphi_h(\cdot,\tau)\|_{\Hdiv}^2 + \|\psi_h(\cdot,\tau)\|_{\Hdiv}^2\bigr) \d \tau \\
	&\ \quad \leq C(\beta) \max\{t^2,t^6 (\eps\mu)^2\} \bigg( \int_0^t \bigl(\|\pa_t j_h(\cdot,\tau)\|_{L^2(\Omega)^3}^2 + \|\pa_t g_h(\cdot,\tau)\|_{L^2(\Omega)^3}^2 \\
	&\ \qquad \qquad+ \|\pa_t^2 \rho_h(\cdot,\tau)\|_{\Hdiv}^2 + \|\pa_t^2 \sigma_h(\cdot,\tau)\|_{\Hdiv}^2 \bigr)\d \tau \bigg) .
	\end{align*}
	provided that $j_h(\cdot,0)=0$, $g_h(\cdot,0)=0$, $\rho_h(\cdot,0)=\pa_t\rho_h(\cdot,0)=0$ and $\sigma_h(\cdot,0)=\pa_t\sigma_h(\cdot,0)=0$.
\end{lemma}
%

\subsection{Interpolation error bounds}
\label{subsection: interpolation error bounds}

We consider the projection of functions on $\Omega$ and $\Gamma$ to continuous piecewise linear finite element functions by interpolation: Let $\P$ denote the operator of piecewise linear (with respect to the triangulation $\calT_h$) and continuous interpolation in~$\Omega$, and let $\Pi_h$ denote the operator of piecewise linear continuous interpolation on~$\Gamma$. Since the normal vector $\nu$ is constant on every face of $\Gamma$, we then have
$$
\Pi_h(\chi\times \nu) = (\Pi_h\chi)\times \nu   \qquad \hbox{for }\ \chi \in C(\Gamma),
$$
which implies that $\Pi_h$ maps $\Hdiv\cap C(\Gamma)$ into $\Hdiv$. Moreover, this yields the very useful relation
\begin{equation}\label{eq: commute}
\Pi_h \gamma_T F = \gamma_T I_h F  \qquad \hbox{for }\ F \in C(\overline\Omega)^3,
\end{equation}
as is seen by noting that
$$
\Pi_h \gamma_T F = \Pi_h (\gamma F \times \nu) = (\Pi_h \gamma F) \times \nu = (\gamma \P F)\times \nu = \gamma_T I_h F.
$$

It is because of \eqref{eq: commute} that we work in the following with interpolation operators rather than orthogonal projections.
We recall the standard results for the interpolation errors.
\begin{lemma}
	\label{lemma: interpolation error}
	There exists a constant $C$, independent of $h$, such that for all $v\in H^2(\Om)^3$,
	\begin{align*}
	\|v - \P v\|_{L^2(\Omega)^3} + h \|\nb(v - \P v)\|_{L^2(\Omega)^{3\times3}} \leq &\ C h^{2} |v|_{H^{2}(\Omega)^3} .
	\end{align*}
\end{lemma}

The following interpolation error estimate is a standard result for boundary element approximations, see \cite{Nedelec}. 

\begin{lemma}
	\label{lemma: interpolation error - boundary}
	There exists a constant $C$, independent of $h$, such that for all $\vphi \in H^{3/2}(\Ga)^3$,
	$$
	\|\vphi - \Pi_h  \vphi\|_{H^{1/2}(\Ga)^3} \leq C h \|\vphi\|_{H^{3/2}(\Ga)^3}.
	$$
\end{lemma}
We remark that for piecewise smooth boundaries just the piecewise $H^{3/2}$ regularity is needed.

For the boundary functions we have the following interpolation error bounds.
\begin{lemma}
	\label{lemma: interpolation error - Calderon}
	There exists a constant $C(t)$, increasing at most polynomially in~$t$ and independent of  $h$, such that for any $t>0$
	\begin{align*}
	&\ \int_0^t \Big\|B(\pa_t) \vect{(I-\Pi_h )\vphi(\cdot,\tau)}{(I-\Pi_h )\psi(\cdot,\tau)}\Big\|_{\Hdiv \times \Hdiv}^2 \d \tau\\
	&\ \leq  C(t) h^2 \int_0^t \bigl(\|\pa_t^2 \vphi(\cdot,\tau)\|_{H^{3/2}(\Ga)^3}^2 + \|\pa_t^2 \psi(\cdot,\tau)\|_{H^{3/2}(\Ga)^3}^2\bigr) \d \tau
	\end{align*}
	for all $\vphi, \psi \in C^2([0,t], \Hdiv\cap H^{3/2}(\Ga)^3)$ with $\vphi(\cdot,0)=\pa_t\vphi(\cdot,0)=0$ and $\psi(\cdot,0)=\pa_t\psi(\cdot,0)=0$.
\end{lemma}
\begin{proof}
	The proof is similar to that of Lemma~7.2 in \cite{leapfrog}: first we bound the action of the blocks of $B(s)$, then we use Plancherel's formula to bound the action of the convolution operator $B(\pa_t)$.
	
	By the boundedness of the boundary integral operators Lemma~\ref{lemma: int operator s bounds}, for $\re s \geq \sigma >0$ we obtain
	\begin{align*}
	&\|V(s)(I-\Pi_h  )\vphi\|_{\Hdiv} \leq \ C|s|^2 \|(I-\Pi_h  )\vphi\|_{\Hdiv} \\
	&\leq \ C|s|^2 \big( \|(I-\Pi_h  )\vphi\|_{\Ht^{-1/2}(\Ga)} + \|\div_\Ga((I-\Pi_h  )\vphi)\|_{H^{-1/2}(\Ga)}\big) \\
	&\leq \ C|s|^2 \,  \|(I-\Pi_h  )\vphi\|_{H^{1/2}(\Ga)^3}\,.
	\end{align*}
	Then, Lemma~\ref{lemma: interpolation error - boundary} yields
	\begin{align*}
	\|V(s)(I-\Pi_h  )\vphi\|_{\Hdiv} \leq &\ C |s|^2 h \|\vphi\|_{H^{3/2}(\Ga)^3}.
	\end{align*}
	A similar estimate holds for the blocks $K(s)$, and so we obtain
	\begin{align*}
	&\ \Big\|B(s) \vect{(I-\Pi_h  )\vphi}{(I-\Pi_h  )\psi}\Big\|_{\Hdiv \times \Hdiv} \leq C|s|^2 h \big( \|\vphi\|_{H^{3/2}(\Ga)^3} + \|\psi\|_{H^{3/2}(\Ga)^3} \big) .
	\end{align*}
	Using Plancherel's formula and causality then yields the stated bound.
	\QED\end{proof}

\subsection{Consistency}
\label{subsection: consistency}

We study the defects (or consistency errors) obtained on inserting the interpolated solution $(\P \E, \P \H, \Pi_h \vphi, \Pi_h \psi)$ into the semidiscrete variational formulation. These defects are defined by
\begin{align*}
&  (d_h^E,w_h) = (\eps \pa_t\P \E,w_h) - \half (\curl_h \P \H,w_h) - \half ( \P \H,\curl_h w_h) -(J,w_h)
\\
& \hphantom{(d_h^E,w_h) = } + \half \la \Pi_h \mu\inv \vphi,\ga_T w_h\ra_\TG
\\
& (d_h^H,z_h) = (\mu \pa_t\P \H,z_h) - \half (\curl_h \P \E,z_h) - \half ( \P \E,\curl_h z_h)
\\
& \hphantom{(d_h^H,z_h) = } + \half \la \Pi_h \psi,\ga_T z_h\ra_\TG \\
&(\xi_h,  d_h^\psi)_\Ga + (\eta_h,d_h^\vphi )_\TG =
\biggl\la \vect{\xi_h}{\eta_h}, B(\pa_t) \vect{\Pi_h\vphi}{\Pi_h\psi} \biggr\ra_\TG
\\
& \hphantom{(\xi_h,  d_h^\psi)_\Ga + (\eta_h,d_h^\vphi )_\TG =} - \half \Big( \la\xi_h,\mu\inv \ga_T \P\E\ra_\TG + \la \eta_h,\ga_T \P\H \ra_\TG \Big)
\end{align*}
for all $w_h,z_h\in V_h$ and $\xi_h,\eta_h\in\Hdiv$.

These defects are bounded as follows.

\begin{lemma}\label{lemma: defect bounds}
	If  the solution of Maxwell's equations \eqref{eq: first order Maxwell} is sufficiently smooth, then the defects satisfy the first-order bounds, for $t>0$,
	\begin{align*}
	&\| d_h^E(t) \|_{L^2(\Omega)^3} \le C\, h, \quad\
	\| d_h^H(t) \|_{L^2(\Omega)^3} \le C\, h,
	\\
	&\biggl(\int_0^t\bigl( \| \pa_t^2 d_h^\psi(\tau) \|_{\Hdiv} ^2 +
	\| \pa_t^2 d_h^\vphi(\tau) \|_{\Hdiv}^2 \bigr)\d \tau\biggr)^{1/2} \le C(t)\, h.
	\end{align*}
	The constant $C(t)$ grows only polynomially with $t$.
\end{lemma}

\begin{proof}
	We begin with the defect $d_h^E$. We have for $w_h \in V_h$,
	\begin{align*}
	(d_h^E,w_h)
	=&\ (\eps \pa_t\P \E,w_h) - (\curl_h \P \H,w_h) -(J,w_h)
	\\
	&+ \half \la \mu\inv \Pi_h \vphi- \ga_T \P H,\ga_T w_h\ra_\TG,
	\end{align*}
	where we used the discrete Green's formula \eqref{eq: Green-h}. Since $\vphi=\mu\ga_T\H$, the boundary term vanishes by the relation \eqref{eq: commute}.
	We further note that
	$\pa_t\P \E=\P\pa_t \E$ and
	$$
	(\curl_h \P \H,w_h) = (\curl \P \H, w_h),
	$$
	because $\P \H$ is a continuous function and so has no jumps on inner faces. The exact solution satisfies  Maxwell's equation and hence
	\begin{align*}
	0 =&\ (\eps \pa_t\P \E,w_h) -  (\curl \P \H,w_h)  -(J,w_h).
	\end{align*}
	Subtracting the two equations therefore yields
	\begin{align*}
	(d_h^E,w_h) =&\ \eps\bigl( \P \pa_t\E-\pa_t \E,w_h\bigr) - \bigl(\curl(\P\H - \H), w_h\bigr).
	\end{align*}
	With the interpolation error bounds of Lemma~\ref{lemma: interpolation error} the right-hand terms are estimated as $O(h)$ times the $L^2(\Omega)$ norm of $w_h$.
	%
	%
	We thus conclude that
	$$
	\| d_h^E \|_{L^2(\Omega)} \le Ch.
	$$
	Similarly we estimate the defect $d_h^H$ for the magnetic equation.
	
	For the boundary   defects $d_h^\psi,d_h^\vphi\in\Psi_h$ we have for all $\xi_h,\eta_h\in \Psi_h$, using the boundary equation,
	$$
	\biggl( \vect{\xi_h}{\eta_h},  \vect{d_h^\psi}{d_h^\vphi} \biggr)_\TG = \biggl\la \vect{\xi_h}{\eta_h},  \vect{\widetilde d_h^\psi}{\widetilde d_h^\vphi} \biggr\ra_\TG
	$$
	where $\widetilde d_h^\psi, \widetilde d_h^\vphi \in \Hdiv$ are given by
	\begin{align*}
	&\vect{\widetilde d_h^\psi}{\widetilde d_h^\vphi} = B(\pa_t) \vect{\Pi_h  \vphi-\vphi}{\Pi_h\psi -   \psi}
	- \half \vect{\ga_T (\P\E -  \E)}{\ga_T (\P\H - \H)} ,
	\end{align*}
	which is bounded by $O(h)$ in the $L^2(0,T;\Hdiv)$ norm by Lemmas~\ref{lemma: interpolation error - Calderon} and \ref{lemma: interpolation error}.
	It then follows that also the defects $d_h^\psi,d_h^\vphi\in\Psi_h$, which are interpolated by $\Psi_h$, are bounded in the same way, using Lemma~\ref{lemma: Hdiv-dual}:
	\begin{align*}
	&\| d_h^\psi \|_{\Hdiv}^2 + \| d_h^\vphi \|_{\Hdiv}^2 \le C \bigl( \| \widetilde d_h^\psi \|_{\Hdiv}^2+ \| \widetilde d_h^\vphi \|_{\Hdiv}^2\bigr).
	\end{align*}
	If we differentiate twice with respect to time before estimating and commute interpolations and time derivatives, this yields the stated bound for the boundary defects.
	\QED\end{proof}

\subsection{Error bound}

\begin{theorem}
	\label{theorem: semidiscrete error bound}
	Assume that the initial data $\E(\cdot,0)$ and $\H(\cdot,0)$ have their support in $\Om$.
	Let the initial values of the semidiscrete problem be chosen as the interpolations of the initial values: $\E_h(\cdot,0)=\P \E(\cdot,0)$ and $\H_h(\cdot,0)=\P \H(\cdot,0)$. If  the solution of Maxwell's equations \eqref{eq: first order Maxwell} is sufficiently smooth, then the error of the dG--BEM semidiscretization \eqref{eq: semidiscrete problem} satisfies, for $t>0$, the first-order error bound
	\begin{align*}
	&\ \eps\|\E_h(\cdot,t) - \E(\cdot,t)\|_{L^2(\Omega)^3}^2 + \mu\|\H_h(\cdot,t) - \H(\cdot,t)\|_{L^2(\Omega)^3}^2 \\
	&\ + \int_0^t \big(\|\vphi_h(\cdot,\tau) - \vphi(\cdot,\tau)\|_{\Hdiv}^2 + \|\psi_h(\cdot,\tau) - \psi(\cdot,\tau)\|_{\Hdiv}^2\big) \d \tau \leq C(t)h^2,
	\end{align*}
	where the constant $C(t)$  grows at most polynomially in $t$.
\end{theorem}
\begin{proof}
	
	We insert the interpolated solution $(\P \E, \P \H, \Pi_h \vphi, \Pi_h \psi)$ into the semi\-discrete variational formulation and apply the stability lemmas, Lemmas~\ref{lemma: semidiscr stability - energy estimate} and \ref{lemma: semidiscr stability - boundary functions}, to the error equations that have the defects in the role of the inhomogeneities. We then use the defect bounds of Lemma~\ref{lemma: defect bounds} to arrive at a first-order error bound for $(\E_h-\P \E, \H_h-\P \H, \vphi_h-\Pi_h \vphi, \psi_h-\Pi_h \psi)$.
	The interpolation error estimates of Lemma~\ref{lemma: interpolation error} and \ref{lemma: interpolation error - Calderon} together with the triangle inequality then complete the proof.
	\QED\end{proof}

\section{Stability results and error bounds for the full discretization}
\label{section: fully discrete results}

Similarly to the semidiscrete case, the stability analysis of the full discretization only depends on the formulation of the fully discrete problem \eqref{eq: full discr - a} and \eqref{eq: full discr - b}, which  again coincides with the acoustic case in form and relevant properties. Hence, the analysis of the full discretization can be carried over directly from \cite[Section 8]{leapfrog}. The original results are again translated into the current functional analytic setting.

\subsection{Stability}

We show stability results under the CFL condition
\begin{equation}\label{eq: CFL}
\dt \|\bfM^{-1/2}\bfD\bfM^{-1/2}\|_2 \leq \sqrt{\eps\mu}.
\end{equation}

The fully discrete electric and magnetic field satisfies the inequality below.
\begin{lemma}
	\label{lemma: fully discr stability - interior}
	Under the CFL condition \eqref{eq: CFL}  and for a stabilization parameter $\alpha\geq1$,  the discrete energy
	\begin{equation*}
	\calE_h^n = \frac \eps 2\|\E_h^n\|_{L^2(\Omega)^3}^2 + \frac{\mu}{4} \Big( \|\H_h^{n+1/2}\|_{L^2(\Omega)^3}^2 + \|\H_h^{n-1/2}\|_{L^2(\Omega)^3}^2 \Big)
	\end{equation*}
	is bounded, at $t=n\dt$, by
	\begin{align*}
	&    \calE_h^n \leq \ C \bigg( \calE_h^0 + \frac{t}{2} \dt \sum_{k=0}^n \Big(\|j_h^{k+1/2}\|_{L^2(\Omega)^3}^2 + \|g_h^{k}\|_{L^2(\Omega)^3}^2 \Big) \\
	&\ + \max\{t^2,t^6\}\dt \sum_{k=0}^n \Big(\|(\pa_t^{\dt})^2 \rho_h^{k+1/2}\|_{\Hdiv}^2 + \|(\pa_t^{\dt})^2 \sigma_h^{k+1/2}\|_{\Hdiv}^2 \Big) \bigg),
	\end{align*}
	where $C>0$ is independent of $h$, $\dt$ and $n$.
\end{lemma}
Using $\H_h^n=\half(\H_h^{n+1/2} + \H_h^{n-1/2})$, the above result also yields a bound on $\|\H_h^n\|^2$.

For the boundary densities we have the following fully discrete estimate.
\begin{lemma}
	\label{lemma: fully discr stability - boundary}
	Under the CFL condition \eqref{eq: CFL}  and for a stabilization parameter $\alpha\geq1$,
	the discrete boundary functions are bounded, at $t=n\dt$, by
	\begin{align*}
	&\ \sum_{k=0}^n \Big( \|\vphi^{k+1/2}\|_{\Hdiv}^2 + \|\psi^{k+1/2}\|_{\Hdiv}^2 \Big) \\
	&\ \hphantom{\leq}\ \leq C \max\{t^2,t^6\} \sum_{k=0}^{n-1} \Big( \|\pa_t^{\dt} j_h^{k}\|_{L^2(\Omega)^3}^2 + \|\pa_t^{\dt} g_h^{k+1/2}\|_{L^2(\Omega)^3}^2 \\
	&\ \hphantom{\leq}\ + \|(\pa_t^{\dt})^2 \rho_h^{k+1/2}\|_{\Hdiv}^2 + \|(\pa_t^{\dt})^2 \sigma_h^{k+1/2}\|_{\Hdiv}^2 \Big)
	\end{align*}
	where $C>0$ is independent of $h$, $\dt$ and $n$.
\end{lemma}


\subsection{Error bound}

The following convergence estimate for the full discretization is then shown in the same way as in the proof of Theorem~9.1 of \cite{leapfrog}, using the consistency errors of the spatial discretization given in Section~\ref{subsection: consistency}, using known error bounds of the leapfrog scheme and convolution quadratures, and applying Lemmas~\ref{lemma: fully discr stability - interior} and \ref{lemma: fully discr stability - boundary}.

\begin{theorem}
	\label{theorem: fully discrete error bound}
	Assume that the initial conditions $\E(\cdot,0)$ and $\H(\cdot,0)$, and the inhomogeneity $\J$ have their supports in $\Om$. Let the initial values of the semidiscrete problem be chosen as the interpolations of the initial values: $\E_h(0)=\P \E(\cdot,0)$ and $\H_h(0)=\P \H(\cdot,0)$. If  the solution of Maxwell's equations \eqref{eq: first order Maxwell} is sufficiently smooth, and
	under the CFL condition \eqref{eq: CFL}  and with a stabilization parameter $\alpha\geq1$,
	the error of the dG--BEM and leapfrog--convolution quadrature discretization \eqref{eq: full discr - a} and \eqref{eq: full discr - b} is bounded, at $t=n\dt$, by
	\begin{align*}
	&\ \eps\|\E_h^n -\E(\cdot,t)\|_{L^2(\Omega)^3}^2 + \mu\|\H_h^n - \H(\cdot,t)\|_{L^2(\Omega)^3}^2 \\
	&\ + \dt \sum_{k=0}^{n-1} \big(\|\vphi_h^{k+1/2} - \vphi(\cdot,t_{k+1/2})\|_{\Hdiv}^2  \\
	&\ \hphantom{ \dt \sum_{k=0}^{n-1} } + \|\bar\psi_h^{k+1/2} - \psi(\cdot,t_{k+1/2})\|_{\Hdiv}^2 \bigr)\leq C(t)(h^2+\dt^4) ,
	\end{align*}
	where the constant $C(t)$ grows at most polynomially in $t$.
\end{theorem}

\section{Conclusion}
We have given a stability and error analysis of semi- and full discretizations of Maxwell's  equations in an interior non-convex domain coupled with time-domain boundary integral equations for transparent boundary conditions. 

A key result for the analysis of this problem is the coercivity estimate of the Calderon operator proved in Lemma~\ref{lemma: coercivity}, which is preserved under trace-space conforming boundary discretizations and translates from the Laplace domain to the continuous-time domain by the operator-valued version of Herglotz' theorem as restated in Lemma~\ref{lemma: Herglotz} and to the convolution quadrature time discretization by Lemma~\ref{lemma: time-discrete coercivity}.

Another important aspect is that the symmetrized weak formulation \eqref{eq: weak formulation}, as first proposed in \cite{abboud2011coupling} for the acoustic wave equation, is preserved under space discretization to yield a finite-dimensional system of the form \eqref{eq: matrix formulation}. In this paper the space discretization is exemplified by a dG discretization in the interior and  continuous boundary elements. Other interior discretizations that are commonly used for Maxwell's equations, such as edge elements, could equally be used as long as they lead to a matrix formulation \eqref{eq: matrix formulation}. Similarly, other trace-space conforming boundary elements such as Raviart--Thomas elements could be used, since they also preserve the coercivity of Lemma~\ref{lemma: coercivity-discrete}.

Once the matrix formulation has the structure  \eqref{eq: matrix formulation} with the coercivity of Lemma~\ref{lemma: coercivity-discrete}, the analysis in \cite{leapfrog} shows stability of the spatial semi-discretization and of the full discretization with a stabilized leapfrog method.

Together with estimates for the consistency error, which we derive in Sections~\ref{subsection: interpolation error bounds} and~\ref{subsection: consistency} in an exemplary way for the particular space discretization considered, we then obtain error bounds for the semi-discretization. Moreover, using known error bounds for the consistency error of the leapfrog method and of the convolution quadrature time discretization on the boundary, we finally obtain error bounds of the full discretization.

We claim no originality on the constituents of the discretization of Maxwell's equation in space and time, in the interior and on the boundary. The novelty of this paper is the stability and error analysis of their coupling. It is remarkable that, in spite of the fundamentally different functional-analytic framework, the stability analysis extends directly from the acoustic to the Maxwell case. 
This becomes possible because we show here that
the coercivity  and the matrix formulation~\eqref{eq: matrix formulation}  of the discretization are of the same type for both Maxwell and the acoustic case. On the other hand, the analysis of the consistency errors depends strongly on the functional-analytic setting and is different for discretizations of Maxwell's equations and the acoustic wave equation.

\begin{acknowledgements} We thank two anonymous referees for their helpful comments.
	We are grateful for the helpful discussions on spatial discretizations with Ralf Hiptmair (ETH Z\"{u}rich) during a BIRS Workshop (16w5071) in Banff.
	This work was supported by the Deutsche Forschungsgemeinschaft (DFG) through SFB 1173.
\end{acknowledgements}



\clearpage

\title{\textbf{\Large Erratum: Stable and convergent fully discrete interior--exterior coupling of Maxwell's equations}}

\titlerunning{Erratum: Interior--exterior coupling for Maxwell's equations}        

\bigskip

\author{\textbf{\large J\"org Nick $\cdot$ 
	Bal\'{a}zs Kov\'{a}cs $\cdot$ 
	\\
	Christian Lubich }
}


\institute{J.~Nick, B.~Kov\'{a}cs, and Ch.~Lubich \at
	Mathematisches Institut, University of T\"{u}bingen,\\
	Auf der Morgenstelle 10, 72076 T\"{u}bingen, Germany \\
	\email{$\{$nick,kovacs,lubich$\}$@na.uni-tuebingen.de}
}

\date{}

\makesecondtitle

\bigskip
\bigskip

\begin{abstract}
	We correct a sign error in the paper \cite{Maxwell} by the second and third authors, noted by the first author. This sign error in the definition of the Calder\'on operator has no effect on the theory presented in  \cite{Maxwell}, but it does affect the implementation of the proposed numerical method.
	
	\keywords{transparent boundary conditions \and boundary integral equations \and Calder\'on operator}
	\subclass{35Q61 \and 65M60 \and 65M38 \and 65M12 \and 65R20}
\end{abstract}

\setcounter{section}{0}

\section{Introduction}

In \cite{Maxwell} we present a time-domain boundary integral formulation of an interior--exterior coupling of Maxwell's equations, with the help of a Calder\'on operator whose coercivity plays a fundamental role in proving the well-posedness of the proposed time-domain boundary integral equations and the stability of the numerical discretization. The definition of the Calder\'on operator contains, however, a sign error, which is corrected here. The effects of this sign error are restricted only to Section~2.3 and formula (3.1) in \cite{Maxwell}, but otherwise all the results of the paper hold unchanged. On the other hand, for the implementation of the method the correct sign is crucial.

\section{The time-harmonic Maxwell's equation and its boundary integral operators}
\label{section:recap helmholtz}
\subsection{Time-harmonic Maxwell's equation and trace operators}

Let us recall the \emph{time-harmonic Maxwell's equation}, obtained as the Laplace transform of the second order Maxwell's equation (with constant permeability $\mu$ and permittivity $\eps$):
\begin{equation}
\label{eq:time-harmonic Maxwell}
\begin{aligned}
\eps\mu s^2u + \curl  \curl u =&\ 0 \qin \R^3 \setminus \Ga ,
\end{aligned}
\end{equation}
where $\Ga$ is the boundary of a bounded piecewise smooth domain (or a finite collection of such domains) $\Om\subset \R^3$, not necessarily convex, with exterior normal $\nu$. The complex parameter $s$ of positive real part is the Laplace transform variable.

In the following we assume appropriate physical units such that 
\begin{equation}\label{epsmu}
\varepsilon\mu=1,
\end{equation}
that is, the wave speed is set to one. 
In the original work~\cite{Maxwell}, the dependence on $\varepsilon\mu$ is unreliable and one should assume \eqref{epsmu}, which just corresponds to a rescaling of time  $t \to t/\sqrt{\eps\mu}\,$ or of frequency $s\to s\sqrt{\eps\mu}$. 

With the scaling \eqref{epsmu}, Eq.~\eqref{eq:time-harmonic Maxwell} becomes the time-harmonic Maxwell's equation $\curl\curl u- \kappa^2 u = 0$  as in \cite{BuffaHiptmair} on setting $s=-i \kappa$.

Analogously to \cite{BuffaHiptmair} the \emph{tangential} and \emph{magnetic} traces are defined by
$$
\gamma_T v = v|_\Ga \times \nu, \andquad \gamma_N v = (s\inv \curl v)|_\Ga \times \nu ,
$$
respectively.
The setting uses the following skew-hermitian pairing on $L^2(\Ga)$:
\begin{equation*}
\la \ga w , \ga v \ra_\Ga = \int_\Ga (\ga \overline w \times \nu) \cdot \ga v \,\d\sigma .
\end{equation*}
The complex conjugation of $w$ was missing in the definition of the pairing in \cite{Maxwell} although it was actually used, e.g. in formula (2.3) and Lemma 3.1 of \cite{Maxwell}.

\subsection{Boundary integral operators}
\label{subsection: boundary integral op}

The functional analytic setting of \cite[Section~2.3]{Maxwell} follows Buffa and Hiptmair \cite{BuffaHiptmair}. The latter paper defines boundary integral operators in the Fourier domain, whereas \cite{Maxwell} uses the Laplace domain (which fits better with convolution quadratures, cf.~\cite[Section~4]{Maxwell}). The sign error occurred while translating the definition of the boundary integral operators and related notions from the Fourier to the Laplace domain. Below we present the correct Laplace domain formulation.

Exactly as in \cite{Maxwell}, following \cite{BuffaHiptmair} and \cite{BallaniBanjaiSauterVeit}, the (electric) \emph{single layer potential} and \emph{double layer potential} for \eqref{eq:time-harmonic Maxwell} are given, for $x \in \R^3\setminus \Ga$, as
\begin{align*}
\calS(s)\vphi(x) = &\ -s \int_\Ga \!\! G(s,x-y) \vphi(y) \d y + s\inv \nb \! \int_\Ga \!\! G(s,x-y) \div_\Ga \vphi(y) \d y , \\
\calD(s)\psi(x) = &\ \curl \int_\Ga \! G(s,x-y) \psi(y) \d y,
\end{align*}
with the fundamental solution $G(s,x) = \tfrac{e^{-s|x|}}{4\pi |x|}$ for $x\in\R^3\setminus\{0\}$ and Re$\,s>0$.

The solution of \eqref{eq:time-harmonic Maxwell} is then given by the correct representation formula:
\begin{equation}
\label{eq:representation formula}
u = - \calS(s)\vphi + \calD(s)\psi,  \qquad x \in \R^3\setminus \Ga .
\end{equation}
In \cite{Maxwell} the first negative sign was erroneously missing.

\noindent The boundary densities in \eqref{eq:representation formula} are given by $\vphi = \jp{\gamma_N u} = \jp{\ga_T (s\inv \curl u)}$ and $\psi = \jp{\gamma_T u}$, where $\jp{\gamma v} = \gamma^- v - \gamma^+ v$ denotes the jump in the boundary traces of the interior domain $\Omega^-$ and the exterior domain $\Omega^+$, while $\av{\gamma v} = \half (\gamma^- v + \gamma^+ v)$ denotes the average. We note  that  there is a sign difference in the jump when comparing \cite{BuffaHiptmair} and \cite{Maxwell}.

Due to the negative sign in the representation formula the correct jump relations are
\begin{align*}
\jp{\gamma_N \circ \calS(s)} =&\ -\hbox{Id}, &  \jp{\gamma_N \circ \calD(s)} =&\ 0,
\\
\jp{\gamma_T \circ \calS(s)} =&\ 0, & \jp{\gamma_T \circ \calD(s)} =&\ \hbox{Id} .
\end{align*}

The boundary integral operators $V$ and $K$ then satisfy the relations
\begin{equation}
\label{eq:boundary integral equations}
\begin{aligned}
V(s) =&\ \av{\gamma_T \circ \calS(s)} = \av{\gamma_N \circ \calD(s)} , \\
K(s) =&\ \av{\gamma_T \circ \calD(s)} = - \av{\gamma_N \circ \calS(s)} . 
\end{aligned}
\end{equation}
In \cite{Maxwell} the negative sign in the last term of the second line was missing. Naturally, this sign  difference does not influence the boundedness of these operators, see \cite[Lemma~2.3]{Maxwell}, based on  \cite[Section~5]{BuffaHiptmair} and \cite{BallaniBanjaiSauterVeit}.

The negative sign in \eqref{eq:boundary integral equations} changes the signs in the expression for the averages of the traces using the operators $V$ and $K$, see \cite[equation~(2.6)]{Maxwell}. The correct relations are:
\begin{equation}
\label{eq:average expressions}
\begin{aligned}
\av{\gamma_T u}
=&\ -\av{ \gamma_T \calS(s)\vphi } + \av{ \gamma_T \calD(s)\psi } \\
=&\ -V(s)\vphi + K(s)\psi , \qquad \andquad \\[2mm]
\av{\gamma_N u}
=&\ -\av{ \gamma_N \calS(s)\vphi } + \av{ \gamma_N \calD(s)\psi } \\
=&\  K(s)\vphi + V(s)\psi .
\end{aligned}
\end{equation}
The negative sign in the first equation was missing in \cite[equation~(2.6)]{Maxwell}.

\section{Coercivity of a Calder\'on operator for the time-harmonic Maxwell's equation}
\label{section:Calderon}
Due to the above formulas, the correct \emph{Calder\'on operator} is given by
\begin{equation}
\label{def-B}
B(s)= \C
\left(
\begin{array}{cc}
- V(s)  &  K(s) \\
-K(s) & -V(s) \\
\end{array}
\right) ,
\end{equation}
with a correct negative sign in the left upper block of $B(s)$ as opposed to \cite[equation~(3.1)]{Maxwell}.

Within the above setting the first equality in the proof of Lemma~3.1 in \cite{Maxwell} stays true: 
For given $\vphi, \psi \in \Hdiv$, we define $u\in H(\curl,\R^3\setminus\Gamma)$ by the representation formula \eqref{eq:representation formula}. We can then express $\vphi$ and $\psi$, see above, by $\vphi = \jp{\gamma_N u} = \jp{\ga_T (s\inv \curl u)}$ and $\psi = \jp{\gamma_T u}$. Then, \eqref{eq:average expressions} and \eqref{def-B} yield
\begin{equation}
B(s) \vect{\vphi}{\psi} = \C \vect{\av{\gamma_T u}}{- \av{\gamma_N u}} .
\end{equation}
\begin{remark}When the scaling \eqref{epsmu} is not imposed, then the corresponding equation is obtained by replacing the argument $s$ with $s\sqrt{\eps\mu}$ in $B(s)$ and in $\gamma_N u=\ga_T (s\inv \curl u)$. We note, however, that with this substitution, the single- and double-layer operators are then scaled differently from those defined in \cite{BallaniBanjaiSauterVeit,Maxwell}.
\end{remark}

It is of crucial importance that in the above setting the Calder\'on operator still satisfies the following coercivity result, with the proof given as in \cite{Maxwell}.
\begin{lemma}[{\cite[Lemma~3.1]{Maxwell}}]
	\label{lemma: coercivity}
	There exists  $\beta>0$ such that the Calder\'on operator \eqref{def-B} satisfies
	\begin{align*}
	\re \biggl\la \vect{\vphi}{\psi}, B(s) \vect{\vphi}{\psi} \biggr\ra_\TG
	\geq \beta \,\m{s}  \Big( \|s\inv\vphi\|_{\Hdiv}^2 + \|s\inv\psi\|_{\Hdiv}^2 \Big)
	\end{align*}
	for $\re s>0$ and for all $\vphi, \psi \in \Hdiv$, with
	$
	\m{s} = \min\{ 1, |s|^{2} \} \re s .
	$
\end{lemma}

Thanks to this coercivity estimate for the Calder\'on operator $B$ defined above in \eqref{def-B}, all the stability and convergence results of \cite{Maxwell} remain valid, since the proofs depend on this coercivity result and not on the particular form of the Calder\'on operator.

\begin{acknowledgements} 
	We thank Jan Bohn for his remark regarding the physical constants.
	The authors are supported by the Deutsche Forschungsgemeinschaft (DFG, German Research Foundation) -- Project-ID 258734477 -- SFB 1173. The work of Bal\'azs Kov\'acs is also supported by the Heisenberg Programme of the Deutsche Forschungsgemeinschaft -- Project-ID 446431602.
\end{acknowledgements}

\end{document}